\documentclass[11pt]{article}\usepackage{amsthm,amsmath,amssymb}
\usepackage[left=0.5in,bottom=1.25in]{geometry}
\usepackage{graphicx}
\usepackage{gensymb}
\usepackage{float}
\usepackage{url}
\usepackage{algorithmic}
\usepackage{algorithm}
\usepackage{titling}
\usepackage{titlesec}
\usepackage{subcaption}
\usepackage{dsfont}
\usepackage{color}
\usepackage{xcolor}
\usepackage{bm}
\usepackage{dsfont}
\usepackage{setspace}
\usepackage{array,multirow}
\usepackage{enumerate}

\titleformat{\subsection}[runin]
  {\normalfont\large\bfseries}{\thesubsection}{1em}{}
\titlelabel{\thetitle.\quad}

\titleformat{\section}[runin]
  {\normalfont\large\bfseries}{\thesection}{1em}{}

\newtheorem*{remark}{Remark}
\newtheorem{theorem}{Theorem}[section]
\newtheorem{lemma}[theorem]{Lemma}

\numberwithin{equation}{section}

\title{Formulation and analysis of a Schur complement method for fluid-structure interaction}
\author{Amy de Castro
\thanks{School of Mathematical and Statistical Sciences, Clemson University, Clemson, SC  29634-0975
    ({\tt agmurda@clemson.edu}). Partially supported by the NSF under grant number DMS-2207971.}
\and
Hyesuk Lee
  \thanks{School of Mathematical and Statistical Sciences, Clemson
        University, Clemson, SC  29634-0975 ({\tt hklee@clemson.edu}).
    Partially supported by the NSF under grant number DMS-2207971.}
\and
Margaret M. Wiecek
  \thanks{School of Mathematical and Statistical Sciences, Clemson
        University, Clemson, SC  29634-0975 ({\tt wmalgor@clemson.edu}).
}    
}
\date{}

\begin{document}

\maketitle

\begin{abstract}

This work presents a strongly coupled partitioned method for fluid structure interaction (FSI) problems based on a monolithic formulation of the system which employs a Lagrange multiplier. We prove that both the semi-discrete and fully discrete formulations are well-posed. To derive a partitioned scheme, a Schur complement equation, which implicitly expresses the Lagrange multiplier and the fluid pressure in terms of the fluid velocity and structural displacement, is constructed based on the monolithic FSI system. Solving the Schur complement system at each time step allows for the decoupling of the fluid and structure subproblems, making the method non-iterative between subdomains. We investigate bounds for the condition number of the Schur complement matrix and present initial numerical results to demonstrate the performance of our approach, which attains the expected convergence rates.

\end{abstract}

\section{Introduction.}
Fluid-structure interaction (FSI) problems model systems in which a deformable structure interacts with a surrounding fluid flow. Applications range from hemodynamics \cite{Gerbeau_2003, Quaini_2009} to aeroelasticity \cite{Zhang_2007}, sedimentation \cite{Wang_2009}, and magneto-hydrodynamic flows \cite{Grigoriadis_2009}. 
The fluid in FSI problems is known to act as an "added mass" on the structure \cite{Causin_2006}. This can pose numerical challenges, particularly when the densities of the fluid and structure are similar, as in hemodynamics. In such cases, the added mass effect can cause certain algorithms to become unstable, presenting extra difficulties in the numerical solution of the problem \cite{Gerbeau_2003, Quaini_2007}.  However, in applications such as aeroelasticity, the added mass effect does not tend to be a concern as the fluid and structure densities are not similar. For a comprehensive overview of numerical approaches to solving FSI problems, we refer the reader to \cite{Hou_2012}.

Methods for numerically solving FSI problems can be categorized as monolithic or partitioned. In a monolithic scheme, the fluid and structural equations, along with the interface conditions between the subdomains, are grouped together in one joint system and solved as a whole. Monolithic methods are highly stable and, therefore, could be preferable in regimes dominated by the added mass effect. They are considered strongly coupled approaches, as the interface conditions are treated implicitly within the system and thus are satisfied exactly at each time step. However, while monolithic methods generally provide stability, they come at a high computational cost. The resulting systems to solve are large, and these methods can be prohibitive to implement since they require the development of a new solver from scratch instead of utilizing already existent solvers for each of the separate physics. For examples of monolithic formulations and solvers, see \cite{Hay_2015, Hubner_2004}.

Because of these costs, partitioned methods are often preferred, as they allow one to take advantage of specialized solvers that already exist for each of the fluid and structural domains. A partitioned algorithm often requires iterations between each of the two subdomains at each time step in order to satisfy interface conditions. Thus, partitioned methods may be distinguished based on their treatment of the interface conditions. Explicit methods often work well in aeroelastic applications in which the fluid is compressible or there is a very low added mass effect \cite{Fernandez_2011}. However, stability issues have been observed numerically and proven analytically for explicit coupling methods, showing that certain scenarios exist in which an explicit method will never be stable, regardless of the time step chosen \cite{Causin_2006}. Some schemes can be stabilized by the addition of perturbations \cite{Burman_2009, Fernandez_2011}.
To address the stability issues, semi-implicit and implicit coupling schemes are often considered, though they can result in a higher computational cost. For an overview of these approaches, see \cite{Fernandez_2011}. 
Semi-implicit algorithms are often based on the Chorin-Temam projection scheme \cite{Chorin_1968, Temam_1968} for incompressible flows \cite{Ballarin_2017, Fernandez_2005, Nonino_2021}. Typically, these involve some explicit projection step to compute fluid velocity and an implicit step that requires iterations between a fluid and structure problem to solve for updated fluid pressure and structural displacement \cite{Fernandez_2005, He_2019}.
On the other hand, for nonlinear FSI models, implicit coupling leads to a nonlinear problem, which is commonly solved by variations on fixed-point and Newton-type algorithms \cite{Gerbeau_2003, Matthies_2002}. 
Many decoupling approaches utilize a Dirichlet-Neumann type strategy, which requires iterations between subdomains at each time step. Time-dependent stability criteria for Dirichlet-Neumann and Neumann-Dirichlet implicit coupling approaches are developed in \cite{Causin_2006}.

The goal of this work is to develop a partitioned method for FSI systems centered around a Schur complement equation that allows one to solve for a Lagrange multiplier (LM) representing the interfacial flux at each time step and thus decouple the subdomain equations. Our approach is based on a monolithic formulation of the FSI problem, which we will show is well-posed. This approach requires no iterations between subdomains, and subdomains are strongly coupled as the two interface conditions are satisfied exactly at each time step. The contribution of this work lies in the new formulation and its analysis; we couple the fluid pressure to an LM stemming from the flux on the interface in the continuous form. We also examine the condition number of the Schur complement matrix.

We briefly consider a few analogous formulations which also utilize Lagrange multipliers, such as Finite Element Tearing and Interconnecting (FETI), and fictitious domain methods. In FETI, LMs are utilized to enforce continuity of a solution variable over subdomain interfaces. This enforcement may be either at interface nodes (classical LMs), or at reference points set up between the subdomains (localized LMs) \cite{Kwak_2014, Lee_2020, Park_2000}. FETI, as originally developed in \cite{Farhat_1994, Farhat_1991}, is mostly viewed as a domain decomposition strategy to allow for parallelization and computational speed-ups \cite{Marcsa_2013}; thus the LMs arise from discretization into multiple subdomains instead of enforcing a continuous interface condition. Variations of FETI  or localized LMs have been developed, and applications to FSI problems can be seen in \cite{ Li_2012, Ross_2008, Walsh_2004}.
In our formulation, the LM is defined at the continuous level to enforce a Neumann boundary condition over the interface, whereas FETI uses LMs to achieve domain decomposition. Unlike FETI, the Schur complement equation arising from our formulation does not involve Boolean matrices \cite{Li_2012}, as the LMs are defined differently, and the equation contains previous solutions for the fluid velocity and structural displacement on the right hand side.

LMs in fictitious domain methods were originally used to enforce conditions on the boundary of the domain \cite{Glowinski_Dir_1994, Glowinski_NS_1994}. In the FSI context, a fictitious domain approach extends the fluid equations into the structural domain \cite{Gerstenberger_2008}. Defining a reference domain for the structure, continuity of the structural velocity with this extended fluid velocity is enforced over the entire reference domain, and this constraint is often enforced by a LM. As the LM is defined on an entire domain, not just a boundary, this is referred to as a distributed LM method. For an overview of fictitious domain methods and distributed LMs, see \cite{Boffi_2015, Glowinski_2007} and references therein. 

Our formulation of the FSI problems shares some similarities with distributed LM-immersed boundary method (DLM-IBM) for FSI problems, as introduced in \cite{Boffi_2015, Boffi_2017}, which utilizes this distributed LM following the idea of a fictitious domain method. Their use of semi-implicit discretization in time gives unconditional stability, as opposed to FE-IBM which is restricted by a Courant-Friedrichs-Lewy (CFL) condition. A rigorous analysis to show the stability of the saddle point formulation and convergence of the discrete solution is provided in \cite{Boffi_2017, Boffi_2022}, and recent work also examines parallel solvers for this formulation \cite{Boffi_2022_Parallel1, Boffi_2023_Parallel2}.
Unlike DLM-IBM, our approach defines the LM only along the interface as opposed to an entire reference domain. This should result in fewer degrees of freedom for the LM. We group the fluid pressure with the LM in the saddle point formulation instead of grouping the pressure with the fluid velocity and structural displacement. In addition, we do not use a fictitious domain approach, so our structure and fluid variables are defined only in their respective subdomains. 

The method we present here is based on the ideas in \cite{de_Castro_2022, Peterson_2019} and \cite{Sockwell_2020}, which, to the best of our knowledge, has only been applied to a transmission problem with a partitioned domain.  Our partitioned approach begins from a monolithic formulation of the problem by using a Lagrange multiplier to treat the interface Neumann condition as a new variable and the continuity of velocity across the interface as an independent equation in the system. We more closely follow \cite{Sockwell_2020}, which discretizes the monolithic system in time first.
The Schur complement equation we derive implicitly expresses a variable representing the Lagrange multiplier and fluid pressure in terms of the fluid velocity and structural displacement. Solving this Schur complement equation allows for the decoupling of the fluid and structural subdomains and their independent solution at each time step, without requiring iterations between them. As opposed to general domain decomposition methods for FSI, this non-iterative approach could be beneficial in allowing for time-independent matrices and a strongly coupled, instead of loosely coupled, scheme. 

In this paper, we show the well-posedness of our formulation and examine the conditioning of the Schur complement matrix. The structure is as follows: in Section \ref{sec:Continuous Model}, we present the model equations and derive the monolithic formulation in the fully continuous setting, showing the well-posedness of the continuous-in-space formulation. In Section \ref{sec:FullyDiscreteModel}, we examine the well-posedness of the fully discrete saddle point problem, which we present in Section \ref{sec:IFRsec} along with the proposed partitioned method. The conditioning of the Schur complement matrix is examined in Section \ref{sec:Conditioning}, and numerical results showing the method's convergence in space and time, as well as the conditioning of the Schur complement system, are shown in Section \ref{sec:Numerical}. Lastly, in Section \ref{sec:Conclusion}, we provide a summary and steps for future work.

\section{Model Equations and Semi-Discrete Model.}\label{sec:Continuous Model}

In this section, we present the semi-discrete monolithic formulation and consider its well-posedness. The FSI problem we study couples an incompressible Newtonian fluid with a linear elastic structure. For ease of developing the partitioned method  and rigorous analysis, we make two simplifications to the general FSI model: utilizing the linear Stokes equations instead of the nonlinear Navier-Stokes equations and considering a fixed domain instead of a moving domain. Future work will consider the general model.

We use the following notation throughout this paper.
Let $\Omega_f, \Omega_s \in \mathbb{R}^d$, for $d=2,3$, refer to the physical fluid and structure domains, respectively, with shared, non-overlapping interface $\gamma$. Outward normal vectors to these domains are represented by $\bm{n_f}$ and $\bm{n_s}$. The Lipschitz continuous boundaries of each domain are denoted by $\Gamma^f, \Gamma^s$, where $\Gamma^f = \Gamma^f_N \cup \Gamma^f_D \cup \gamma$ and $\Gamma^s = \Gamma^s_N \cup \Gamma^s_D \cup \gamma$. The subscripts ``N'' and ``D'' refer to portions of the boundary on which Neumann and Dirichlet boundary conditions are respectively defined. Assume the measure of $\Gamma_N$ and $\Gamma_D$ are nonzero on both domains. We denote the temporal domain by $(0,T)$, where $T$ is a given final time. 

The notation $(\cdot,\cdot) = (\cdot,\cdot)_{\omega}$ is employed to represent the $L^2$ inner product on a subdomain or interface $\omega$ and $\langle \cdot, \cdot \rangle_\omega$ to denote the duality pairing between $H^{-1}(\omega)$ and $H^1(\omega)$ or $H^{-1/2}(\omega)$ and $H^{1/2}(\omega)$ on the appropriate subdomain or interface $\omega$. Additionally, we utilize the following $H^1$ norm for a generic function $\bm{g}$ in $\omega$: $||\bm{g}||^2_{1,\omega} := ||\bm{g}||_{0,\omega}^2 + ||D(\bm{g})||_{0,\omega}^2 $.

\noindent The continuous FSI system can be written as follows.

\noindent \textit{Find  $\bm{u} \in \Omega_f \times (0,T) \mapsto \mathbb{R}^d, \hspace{1mm} p \in \Omega_f \times (0,T) \mapsto \mathbb{R}, \hspace{1mm} \bm{\eta} \in \Omega_s \times (0,T) \mapsto \mathbb{R}^d $ such that}
\begin{align}
\rho_f  \frac{\partial \bm{u}}{\partial t} - 2 \nu_f \nabla \cdot D(\bm{u}) + \nabla p &= \bm{f_f} \hspace{5mm} \text{ in }\Omega_f \times (0,T),  \label{StokesMom} \\
\nabla \cdot \bm{u} &= 0 \hspace{8mm} \text{ in }\Omega_f \times (0,T),  \label{StokesMass} \\
\rho_s \frac{\partial^2 \bm{\eta}}{\partial t^2} - 2\nu_s \nabla \cdot D(\bm{\eta}) - \lambda \nabla (\nabla \cdot \bm{\eta}) &= \bm{f_s} \hspace{6mm} \text{ in }\Omega_s \times (0,T). \label{Structure}
\end{align}
Appropriate initial conditions are given, as well as both Dirichlet and Neumann boundary conditions.
\begin{align}
\begin{split}\label{noninterBCs}
(2\nu_f D(\bm{u}) - p) \cdot \bm{n_f} &= \bm{u_N} \hspace{5mm} \text{ on } \Gamma_N^f  \times (0,T), \hspace{8mm}
\bm{u} = \bm{0} \hspace{5mm} \text{ on } \Gamma_D^f \times (0,T), \\
(2\nu_s D(\bm{\eta}) + \lambda (\nabla \cdot \bm{\eta})) \cdot \bm{n_s} &= \bm{\eta_N} \hspace{5mm} \text{ on } \Gamma_N^s \times (0,T), \hspace{8mm}
\bm{\eta} = \bm{0} \hspace{5mm} \text{ on } \Gamma_D^s \times (0,T).
\end{split}
\end{align}
Here, $\bm{u}(\bm{x},t)$ denotes the fluid velocity, $p(\bm{x},t)$ the fluid pressure, and $\bm{\eta}(\bm{x},t)$ the displacement of the structure. The functions $\bm{f_f}, \bm{f_s}, \bm{u_N}, \bm{\eta_N}$ are given body forces and Neumann boundary conditions. The operator $D(\cdot)$ is the rate of strain tensor, defined as $D(\bm{v}) := \frac{1}{2}(\nabla \bm{v} + (\nabla \bm{v})^T)$. The constants in equations \eqref{StokesMom} - \eqref{Structure} are $\rho_f$ denoting the fluid density, $\rho_s$ the structure density, and $\nu_f$ the fluid viscosity. For the structure equation \eqref{Structure}, the Lam\'e parameters are represented by $\nu_s, \lambda$.  All constants are assumed positive. \\
\noindent The conditions for $\bm{u}, \bm{\eta}$ given on the interface $\gamma$ are obtained from enforcing continuity of velocity and continuity of the stress force.
\begin{align}
\frac{\partial \bm{\eta}}{\partial t} &= \bm{u} \hspace{5mm} \text{ on } \gamma \times (0,T), \label{Inter1}\\
(2\nu_f D(\bm{u}) - p) \cdot \bm{n_f} &= -(2\nu_s D(\bm{\eta}) + \lambda (\nabla \cdot \bm{\eta})) \cdot \bm{n_s} \hspace{5mm} \text{ on } \gamma \times (0,T). \label{Inter2}
\end{align}
Define the following spaces: 
\begin{align}\label{ContSpaces}
\begin{split}
U &:=\{\bm{v} \in \bm{H^1}(\Omega_f): \bm{v} = \bm{0} \text{ on } \Gamma_D^f\}, \hspace{5mm}
Q :=L^2(\Omega_f), \hspace{5mm}
X := \{\bm{\varphi} \in \bm{H^1}(\Omega_s): \bm{\varphi} = \bm{0} \text{ on } \Gamma_D^s\}.
\end{split}
\end{align}

Let $\bm{g}$ be a Lagrange multiplier associated with the interface condition \eqref{Inter2}, such that on $\gamma$,\\ $\bm{g} := (2\nu_f D(\bm{u}) - p) \cdot \bm{n_f}$. This implies that $(2\nu_s D(\bm{\eta}) + \lambda (\nabla \cdot \bm{\eta})) \cdot \bm{n_s} = -\bm{g}$. Take $\Lambda := H^{-1/2}(\gamma)$ to be the Lagrange multiplier space.
To derive a weak formulation of \eqref{StokesMom}-\eqref{Inter2}, we multiply \eqref{StokesMom}-\eqref{Structure} by appropriate test functions $\bm{v} \in U, q \in Q, \bm{\varphi} \in X$, multiply the constraint equation \eqref{Inter1} by a test function $\bm{s} \in \Lambda$, and integrate each equation. We also substitute the Lagrange multiplier $\bm{g}$ in the relevant boundary condition terms to obtain the following weak form.

\noindent \textit{Find $\bm{u} \in U, p \in Q, \bm{\eta} \in X, \bm{g} \in \Lambda$ such that}
\begin{align}
\begin{split} \label{WeakForm}
\rho_f \Big( \frac{\partial \bm{u}}{\partial t}, \bm{v} \Big)_{\Omega_f} + 2 \nu_f \left( D(\bm{u}), D(\bm{v}) \right)_{\Omega_f} - (p, \nabla \cdot \bm{v})_{\Omega_f} - \langle\bm{g},\bm{v}\rangle_\gamma  = \langle\bm{f_f},\bm{v}\rangle_{\Omega_f} + \langle\bm{u}_N, \bm{v}\rangle_{\Gamma_N^f} \hspace{3mm} \forall \bm{v} \in U, \\
( \nabla \cdot \bm{u}, q )_{\Omega_f} = 0 \hspace{3mm} \forall q \in Q, \\
\rho_s \Big( \frac{\partial^2 \bm{\eta}}{\partial t^2}, \bm{\varphi} \Big)_{\Omega_s} +  2 \nu_s \left( D(\bm{\eta}), D(\bm{\varphi}) \right)_{\Omega_s} + \lambda ( \nabla \cdot \bm{\eta}, \nabla \cdot \bm{\varphi} )_{\Omega_s} + \langle\bm{g}, \bm{\varphi}\rangle_\gamma  = \langle\bm{f_s}, \bm{\varphi}\rangle_{\Omega_s} + \langle\bm{\eta}_N, \bm{\varphi}\rangle_{\Gamma_N^s} \hspace{3mm} \forall \bm{\varphi} \in X, \\
\Bigl \langle \frac{\partial \bm{\eta}}{\partial t}, \bm{s} \Bigr \rangle_\gamma - \langle\bm{u}, \bm{s}\rangle_\gamma = 0 \hspace{3mm} \forall \bm{s} \in \Lambda.
\end{split}
\end{align}

\noindent Define the spaces $Y = U \times X$, and $Z = Q \times \Lambda$, with corresponding norms:
\begin{align*}
    ||\bm{y}||_Y^2 = ||(\bm{u},\bm{\eta})||_Y^2 &:= ||\bm{u}||_{1,\Omega_f}^2 + ||\bm{\eta}||_{1,\Omega_s}^2, \\
    ||\bm{z}||_Z^2 = ||(p,\bm{g})||_Z^2 &:= ||p||_{0,\Omega_f}^2 + ||\bm{g}||^2_{-1/2,\gamma}.
\end{align*}
Using first order backward difference approximations for the time derivatives $\dot{\bm{u}} :=\frac{\partial \bm{u}}{\partial t}$ and $\ddot{\bm{\eta}}:=\frac{\partial^2 \bm{\eta}}{\partial t^2}$,

\begin{align*}
\dot{\bm{u}} &= \frac{\bm{u}^{n+1} - \bm{u}^n}{\Delta t}, \hspace{8mm}
\ddot{\bm{\eta}} = \frac{\bm{\eta}^{n+1} - 2\bm{\eta}^n + \bm{\eta}^{n-1}}{\Delta t^2}, 
\end{align*}
we discretize \eqref{WeakForm} in time using Backward Euler:
\begin{align}\label{WeakFormTimeDisc}
\begin{split}
\rho_f \left( \bm{u}^{n+1}, \bm{v} \right)_{\Omega_f} + \Delta t  2 \nu_f &\left( D(\bm{u}^{n+1}), D(\bm{v}) \right)_{\Omega_f} - \Delta t(p^{n+1}, \nabla \cdot \bm{v})_{\Omega_f}  - \Delta t\langle\bm{g}^{n+1},\bm{v}\rangle_\gamma  \\
&= \Delta t\langle \bm{f_f}^{n+1},\bm{v}\rangle_{\Omega_f}  + \Delta t\langle\bm{u}_N^{n+1}, \bm{v}\rangle_{\Gamma_N^f} + \rho_f \left( \bm{u}^n, \bm{v} \right)_{\Omega_f} \hspace{3mm} \forall \bm{v} \in U, \\
( \nabla \cdot \bm{u}^{n+1}, q )_{\Omega_f} &= 0 \hspace{3mm} \forall q \in Q, \\
\frac{1}{\Delta t} \rho_s \left( \bm{\eta}^{n+1}, \bm{\varphi} \right)_{\Omega_s} +  2 \Delta t \nu_s &\left( D(\bm{\eta}^{n+1}), D(\bm{\varphi}) \right)_{\Omega_s} + \Delta t \lambda ( \nabla \cdot \bm{\eta}^{n+1}, \nabla \cdot \bm{\varphi} )_{\Omega_s} + \Delta t \langle\bm{g}^{n+1}, \bm{\varphi}\rangle_\gamma  \\&= \Delta t\langle \bm{f_s}^{n+1}, \bm{\varphi}\rangle_{\Omega_s} + \Delta t\langle\bm{\eta}_N^{n+1}, \bm{\varphi}\rangle_{\Gamma_N^s}+ \frac{1}{\Delta t}\rho_s \left( 2\bm{\eta}^n - \bm{\eta}^{n-1}, \bm{\varphi} \right)_{\Omega_s}  \hspace{3mm} \forall \bm{\varphi} \in X, \\
\frac{1}{\Delta t} \langle \bm{\eta}^{n+1}, \bm{s} \rangle_\gamma - \langle\bm{u}^{n+1}, \bm{s}\rangle_\gamma &= \frac{1}{\Delta t} \langle \bm{\eta}^n, \bm{s} \rangle_\gamma \hspace{3mm} \forall \bm{s} \in \Lambda.
\end{split}
\end{align}

\noindent This formulation suggests the bilinear forms $a(\bm{y_1},\bm{y_2}): \hspace{1mm} Y \times Y \mapsto \mathbb{R}$ and $b(\bm{y},\bm{z}): \hspace{1mm} Y \times Z \mapsto \mathbb{R},$ where
\begin{align}\label{eq:bilinearForms}
\begin{split}
a\left( (\bm{u},\bm{\eta}), (\bm{v}, \bm{\varphi}) \right)  &:= \rho_f (\bm{u},\bm{v})_{\Omega_f} + \Delta t 2\nu_f  (D(\bm{u}),D(\bm{v}))_{\Omega_f} \\& \hspace{5mm} + \rho_s(\bm{\eta},\bm{\varphi})_{\Omega_s} + \Delta t^2 2 \nu_s (D(\bm{\eta}),D(\bm{\varphi}))_{\Omega_s} + \Delta t^2 \lambda (\nabla \cdot \bm{\eta}, \nabla \cdot \bm{\varphi})_{\Omega_s}, \\
b\left((\bm{v}, \bm{\phi}), (q, \bm{s}) \right) &:= b_\gamma\left( (\bm{v},\bm{\phi}), \bm{s}\right) + b_p(\bm{v},q),
\end{split}
\end{align}
and $b_\gamma\left( (\bm{v},\bm{\phi}) ,s\right) := \langle\bm{\phi},\bm{s}\rangle_\gamma -\langle\bm{v},\bm{s}\rangle_\gamma, \hspace{3mm}
b_p(\bm{v},q) := -(\nabla \cdot \bm{v}, q)_{\Omega_f}.$
Then the FSI problem \eqref{WeakFormTimeDisc} may be stated as the following saddle point system. 
\textit{
Find $\Big((\bm{u}^{n+1}, \bm{\eta}^{n+1}), (p^{n+1},\bm{g}^{n+1})\Big) \in Y \times Z$ such that}
\begin{align}\label{eq:SaddlePt}
\begin{split}
a((\bm{u}^{n+1},\frac{1}{\Delta t}\bm{\eta}^{n+1}),(\bm{v},\bm{\varphi})) + b((\bm{v},\bm{\varphi}),(\Delta t p^{n+1},\Delta t \bm{g}^{n+1})) &= \mathcal{F}_1(\bm{v},\bm{\varphi}) \hspace{5mm} \forall \hspace{1mm} (\bm{v},\bm{\varphi}) \in Y, \\
b((\bm{u}^{n+1},\frac{1}{\Delta t}\bm{\eta}^{n+1}),(q,\bm{s})) &= \mathcal{F}_2(q,\bm{s}) \hspace{5mm} \forall \hspace{1mm} (q,\bm{s}) \in Z,
\end{split}
\end{align}
where $\mathcal{F}_1(\bm{v},\bm{\varphi}) = \Delta t\langle \bm{f_f}^{n+1},\bm{v}\rangle_{\Omega_f}  + \Delta t\langle\bm{u}_N^{n+1}, \bm{v}\rangle_{\Gamma_N^f} + \rho_f \left( \bm{u}^n, \bm{v} \right)_{\Omega_f} + \Delta t  \langle \bm{f_s}^{n+1}, \bm{\varphi}\rangle_{\Omega_s} + \Delta t \langle\bm{\eta}_N^{n+1}, \bm{\varphi}\rangle_{\Gamma_N^s}+ \frac{\rho_s}{\Delta t} \left( 2\bm{\eta}^n - \bm{\eta}^{n-1}, \bm{\varphi} \right)_{\Omega_s}, $ and $\mathcal{F}_2(q,\bm{s}) = \frac{1}{\Delta t}\langle\bm{\eta}^n,\bm{s}\rangle_\gamma$.\\

Next, we examine the well-posedness of the semi-discrete model \eqref{eq:SaddlePt}. Our proofs follow the structure of the proofs from \cite{Gunzburger_1992} for a Navier-Stokes problem which uses a Lagrange multiplier to implement inhomogeneous Dirichlet boundary conditions.
\noindent For the saddle point problem \eqref{eq:SaddlePt} to be well-posed, we need to show the coercivity of $a(\cdot,\cdot)$ and the inf-sup condition between $Y$ and $Z$ for the form $b(\cdot,\cdot)$ (see, for example, \cite{Braess_1997}, page 127, Thm. 4.3).

\noindent The coercivity of $a(\cdot,\cdot)$ can be easily shown.
Let $a((\bm{u},\bm{\eta}),(\bm{v},\bm{\varphi}))$ be defined as in \eqref{eq:bilinearForms}. Then
\begin{align}
    a\left( (\bm{v},\bm{\varphi}), (\bm{v}, \bm{\varphi}) \right)  &:= \rho_f (\bm{v},\bm{v})_{\Omega_f} + \Delta t 2\nu_f (D(\bm{v}),D(\bm{v}))_{\Omega_f} + \rho_s(\bm{\varphi},\bm{\varphi})_{\Omega_s} + \Delta t^2 2 \nu_s (D(\bm{\varphi}),D(\bm{\varphi}))_{\Omega_s} \nonumber \\
    & \hspace{1in} + \Delta t^2 \lambda (\nabla \cdot \bm{\varphi}, \nabla \cdot \bm{\varphi})_{\Omega_s} \nonumber\\
    &= \rho_f||\bm{v}||_0^2 + \Delta t 2 \nu_f ||D(\bm{v})||_0^2 + \rho_s ||\bm{\varphi}||_0^2 + \Delta t^2 2 \nu_s||D(\bm{\varphi})||_0^2 + \Delta t^2 \lambda ||\nabla \cdot \bm{\varphi}||_0^2 \nonumber\\
    &\geq \min \{\rho_f, \Delta t 2\nu_f \} ||\bm{v}||_1^2 + \min \{\rho_s, \Delta t^2 2 \nu_s\} ||\bm{\varphi}||_1^2 +  \Delta t^2 \lambda ||\nabla \cdot \bm{\varphi}||_0^2 \nonumber \\
    &\geq \min\{\rho_f,\Delta t 2\nu_f , \rho_s, \Delta t^2 2\nu_s\}||(\bm{v},\bm{\varphi})||_Y^2.
    \label{a-coercivity}
\end{align}

\noindent Now, we prove the inf-sup condition given in Theorem \ref{thm:ContInfSup} between $Y$ and $Z$. 
 First, define the space $$U_{\gamma 0} := \{\bm{v} \in U : \bm{v} = \bm{0} \text{ on } \gamma \}.$$  It has been shown (see \cite{Du_2003}, \cite{Charina_2004}) that an inf-sup condition is satisfied between the spaces $U_{\gamma 0}$ and $Q$. 
\begin{equation}\label{KQinfsup}
\text{There exists a } \bm{v}^*\neq 0 \in U_{\gamma 0} \text{ and } \beta^* > 0 \text{ such that } (\nabla \cdot \bm{v}^*,q) \geq \beta^* ||\bm{v}^*||_1 ||q||_0 \hspace{5mm} \forall q \in Q.
\end{equation}
Since $U_{\gamma 0}$ is a subspace of U, the above inf-sup condition also holds between U and Q. \\
Next, the divergence operator, $\bm{w} \mapsto \nabla \cdot \bm{w}$ has been shown to map $U_{\gamma 0}$ onto $Q$ \cite{Charina_2004}.
Thus, for any $q \in Q$, we may find a solution $\bm{w} \in U_{\gamma 0}$ to the problem 
\begin{align}\label{wExist-1}
\nabla \cdot \bm{w} &= q \hspace{5mm} \text{ in }\Omega_f .
\end{align}
The following lemma is needed for the proof of the inf-sup condition.
\begin{lemma}\label{lemma:bpExist}
For a given $\tilde{q} \in Q, \bm{s}^* \in H^{1/2}(\gamma)$, the following problem has a solution $\tilde{\bm{u}} \in U$:
\begin{align}\label{eq:bp}
\begin{split}
b_p(\bm{\tilde{u}},q) &= (\tilde{q},q)_{\Omega_f} \hspace{3mm} \forall q \in Q, \\
\bm{\tilde{u}}|_\gamma &= -\bm{s}^*,
\end{split}
\end{align} 
with
\begin{equation}\label{eq:uBound}
||\bm{\tilde{u}}||_1 \leq C\Big(||\tilde{q}||_0 + ||\bm{s}^*||_{1/2,\gamma} \Big).
\end{equation} 
\end{lemma}

\begin{proof}
Let $\tilde{q} \in Q, \bm{s}^* \in H^{1/2}(\gamma)$ be given. From the inf-sup condition between U and Q and the surjectiveness of the divergence operator stated previously, we may find a $\tilde{\bm{u}}_1 \in U$ satisfying 
$\nabla \cdot \tilde{\bm{u}}_1  = \tilde{q} \in \Omega_f$, which implies 
\begin{align}
(\nabla \cdot \tilde{\bm{u}}_1 ,q) &= (\tilde{q},q) \hspace{5mm} \forall q \in Q,  \label{wExist} \\
(\nabla \cdot \tilde{\bm{u}}_1 ,q) &\geq \beta^* ||\tilde{\bm{u}}_1 ||_1 ||q||_0 \hspace{5mm} \forall q \in Q. \label{infsupKQ}
\end{align}
By these properties of $\tilde{\bm{u}}_1$, \eqref{wExist} and \eqref{infsupKQ},  we see
$$
||\tilde{q}||_0^2 = (\tilde{q},\tilde{q}) 
= (\nabla \cdot \tilde{\bm{u}}_1, \tilde{q}) 
\geq \beta^* ||\tilde{\bm{u}}_1||_1 ||\tilde{q}||_0,
$$
and, therefore,
\begin{equation} \label{qboundu1}
||\tilde{q}||_0 \geq  \beta^* ||\tilde{\bm{u}}_1||_1.
\end{equation}

Now, define $\overline{\bm{w}} \in U$ such that $\overline{\bm{w}}|_\gamma = -\bm{s}^* + \tilde{\bm{u}}_1|_\gamma$. 
Since div$(\overline{\bm{w}}) \in Q$, the surjectivity of the divergence operator gives the existence of a  $\bm{v} \in U_{\gamma 0}$ such that 
\begin{align}\label{vprops}
\begin{split}
(\nabla \cdot \bm{v},q) = (\nabla \cdot \overline{\bm{w}}, q) \hspace{3mm} \forall q \in Q, \hspace{2mm} \text{ and } \hspace{2mm}
||\bm{v}||_1 &\leq C ||\nabla \cdot \overline{\bm{w}}||_0.
\end{split}
\end{align} 
Take $\tilde{\bm{u}}_2: = \overline{\bm{w}} - \bm{v}$. Note that $\tilde{\bm{u}}_2 \in U$. Then, 

\begin{align}
\begin{split}\label{u2props}
(\nabla \cdot \tilde{\bm{u}}_2, q) &= (\nabla \cdot \overline{\bm{w}}, q) - (\nabla \cdot \bm{v}, q) = 0 \hspace{5mm} \forall q \in Q, \\
\tilde{\bm{u}}_2|_\gamma &= \overline{\bm{w}}|_\gamma - \bm{v}|_\gamma = -\bm{s}^* + \tilde{\bm{u}}_1|_\gamma.
\end{split} 
\end{align}
as $\bm{v} \in U_{\gamma 0}$.
The norm of $\tilde{\bm{u}}_2$ can be bounded as follows using \eqref{vprops}:
\begin{equation}\label{normu2bound}
||\tilde{\bm{u}}_2||_1 \,\leq \, ||\overline{\bm{w}}||_1 + ||\bm{v}||_1 
\, \leq \, ||\overline{\bm{w}}||_1 + C ||\nabla \cdot \overline{\bm{w}}||_0   
\,\leq \, C ||\overline{\bm{w}}||_1.
\end{equation}
Taking the infimum of \eqref{normu2bound} over all $\overline{\bm{w}} \in U$ satisfying $\overline{\bm{w}}|_\gamma = -\bm{s}^* + \tilde{\bm{u}}_1|_\gamma$ gives:

\begin{equation}\label{normu2bounduse}
||\tilde{\bm{u}}_2||_1 \leq C||-\bm{s}^* + \tilde{\bm{u}}_1||_{1/2,\gamma}.
\end{equation}
Now, define $\tilde{\bm{u}} \in U$ to be $\tilde{\bm{u}}:= -\tilde{\bm{u}}_1 + \tilde{\bm{u}}_2$. By \eqref{wExist} and \eqref{u2props}, we have
\begin{align*}
(\nabla \cdot \tilde{\bm{u}}, q) &= (\nabla \cdot (-\tilde{\bm{u}}_1), q) + (\nabla \cdot \tilde{\bm{u}}_2, q) = -(\tilde{q},q) \hspace{5mm} \forall q \in Q,  \\
\tilde{\bm{u}}|_\gamma &= -\tilde{\bm{u}}_1|_\gamma + \tilde{\bm{u}}_2|_\gamma =  -\tilde{\bm{u}}_1|_\gamma - \bm{s}^* +  \tilde{\bm{u}}_1|_\gamma = -\bm{s}^*. 
\end{align*}
Also, \eqref{qboundu1} and \eqref{normu2bounduse} yield 
\begin{align*}
||\tilde{\bm{u}}||_1 &\leq ||\tilde{\bm{u}}_1||_1 + || \tilde{\bm{u}}_2||_1 \leq \Big(\frac{1}{ \beta^*} ||\tilde{q}||_0 + C||-\bm{s}^* + \tilde{\bm{u}}_1||_{1/2,\gamma} \Big)\\
&\leq C (||\tilde{q}||_0 + ||\bm{s}^*||_{1/2,\gamma} + ||\tilde{\bm{u}}_1||_{1/2,\gamma} ) \leq C (||\tilde{q}||_0 + ||\bm{s}^*||_{1/2,\gamma} + ||\tilde{\bm{u}}_1||_{1} )\\
&\leq C ( ||\tilde{q}||_0 + ||\bm{s}^*||_{1/2,\gamma}). 
\end{align*}
\end{proof}

\begin{theorem}\label{thm:ContInfSup}
There exists a positive constant $\beta$ such that
\begin{align}\label{infsup}
\underset{\bm{0} \neq \bm{y} \in Y}{\sup} \frac{b(\bm{y}; \bm{z})}{||\bm{y}||_Y} \geq \beta ||\bm{z}||_Z > 0 \hspace{3mm} \forall \bm{z} \in Z.
\end{align}
\end{theorem}
\begin{proof}
Let $\tilde{\bm{z}} = (\tilde{q}; \tilde{\bm{s}}) \in Z$ be given. Since $\Lambda$ is a Hilbert space, by the Riesz Representation Theorem, we may find an $\bm{s^*} \in \Lambda^* = H^{1/2}(\gamma)$ with the following properties:
\begin{align}
||\bm{s^*}||_{1/2,\gamma} &= ||\tilde{\bm{s}}||_{-1/2,\gamma} \label{eq:sNorm}\\
\langle \bm{\tilde{s}},\bm{\theta}\rangle_\gamma &= (\bm{s^*}, \bm{\theta})_{1/2,\gamma} \hspace{3mm} \forall \bm{\theta} \in \Lambda^* = H^{1/2}(\gamma). \label{eq:sInnerProd}
\end{align}
Also, we can choose a $\bm{\tilde{\eta}} \in X$ such that 
\begin{equation} \label{eq:etaTrace}
\bm{\tilde{\eta}}\big|_\gamma = \bm{s}^* \quad\text{and} \quad 
||\bm{\tilde{\eta}}||_1 \leq C_1 ||\bm{s}^*||_{1/2,\gamma}. 
\end{equation}
By Lemma \ref{lemma:bpExist}, we may choose a $\tilde{\bm{u}} \in U$ such that
\begin{align}
b_p(\tilde{\bm{u}},\tilde{q}) &= -(\nabla \cdot \tilde{\bm{u}}, q) = (\tilde{q},q) \hspace{5mm} \forall q \in Q, \label{tildeueqq} \\
\tilde{\bm{u}}|_\gamma &= -\bm{s}^*, \label{tildeugam}\\
||\tilde{\bm{u}}||_1 &\leq C(||\tilde{q}||_0 + ||\bm{s}^*||_{1/2,\gamma}). \label{tildeubound}
\end{align}
Note that \eqref{eq:etaTrace} and \eqref{tildeubound} imply that 
\begin{equation}\label{eq:uetaBound}
||\bm{\tilde{u}}||_1 + ||\bm{\tilde{\eta}}||_1 \leq \overline{C} \Big(||\tilde{q}||_0 + ||\bm{s}^*||_{1/2,\gamma}\Big).
\end{equation}
Then, using \eqref{eq:sNorm} - \eqref{tildeugam} and \eqref{eq:uetaBound}, 
\begin{align*}
b(\tilde{\bm{u}},\tilde{\bm{\eta}}; \tilde{q},\bm{\tilde{s}}) &= b_\gamma(\tilde{\bm{u}},\tilde{\bm{\eta}},\bm{\tilde{s}}) + b_p(\tilde{\bm{u}}; \tilde{q}) = \langle\bm{\tilde{\eta}},\bm{\tilde{s}}\rangle_\gamma -  \langle\bm{\tilde{u}},\bm{\tilde{s}}\rangle_\gamma + (\tilde{q},\tilde{q})_{\Omega_f} \\ 
&= (\bm{\tilde{\eta}},\bm{s^*})_{1/2,\gamma} - (\bm{\tilde{u}},\bm{s^*})_{1/2,\gamma} + ||\tilde{q}||_0^2  \\
&= ||\bm{s}^*||^2_{1/2,\gamma} + ||\bm{s}^*||^2_{1/2,\gamma} + ||\tilde{q}||_0^2  \\
&\geq  ||\bm{s}^*||^2_{1/2,\gamma} + 1/2 \Big( ||\bm{s}^*||_{1/2,\gamma} + ||\tilde{q}||_0\Big) \Big( ||\bm{s}^*||_{1/2,\gamma} + ||\tilde{q}||_0\Big)  \\
&\geq ||\bm{s}^*||^2_{1/2,\gamma} + \frac{1/2}{\overline{C}} \Big( ||\tilde{\bm{u}}||_{1} + ||\tilde{\bm{\eta}}||_1\Big) \Big( ||\bm{s}^*||_{1/2,\gamma} + ||\tilde{q}||_0\Big) \\
&= ||\bm{s}^*||^2_{1/2,\gamma} + \frac{1/2}{\overline{C}} \Big( ||\tilde{\bm{u}}||_{U} + ||\tilde{\bm{\eta}}||_X\Big) \Big( ||\bm{\tilde{s}}||_{-1/2,\gamma} + ||\tilde{q}||_0\Big)  \\
&\geq  \frac{1/2}{\overline{C}} \Big( ||\tilde{\bm{u}}||_{U} + ||\tilde{\bm{\eta}}||_X\Big) \Big( ||\bm{\tilde{s}}||_{-1/2,\gamma} + ||\tilde{q}||_0\Big) \\
&\geq \frac{1/2}{\overline{C}} \Big( ||\tilde{\bm{u}}||^2_{U} + ||\tilde{\bm{\eta}}||^2_X\Big)^{1/2} \Big( ||\bm{\tilde{s}}||^2_{-1/2,\gamma} + ||\tilde{q}||^2_0\Big)^{1/2} \\
&= \frac{1/2}{\overline{C}} ||\tilde{\bm{u}}; \tilde{\bm{\eta}}||_Y ||\tilde{q}; \tilde{\bm{s}}||_Z.
\end{align*}
\noindent Thus, we see that 
\begin{align*}
\frac{b(\tilde{\bm{u}},\tilde{\bm{\eta}}; \tilde{q},\tilde{\bm{s}})}{||\tilde{\bm{u}}; \tilde{\bm{\eta}}||_Y}& \geq \beta ||\tilde{q}; \tilde{\bm{s}}||_Z.
\end{align*}
As $\tilde{z} = (\tilde{q},\tilde{\bm{s}})$ was arbitrary, taking the supremum over $0 \neq (\bm{u},\bm{\eta}) \in Y$, the inf-sup condition \eqref{infsup} holds.
\end{proof}

Using Theorem \ref{thm:ContInfSup}, 
\eqref{a-coercivity} and the existence theory for a saddle point problem, we establish the well-posedness of the semi-discrete system \eqref{eq:SaddlePt}.
\begin{theorem} \label{semi-discrete-well-posedness}
The semi-discrete system \eqref{eq:SaddlePt} has a unique solution $\left( (\bm{u}^{n+1}, \bm{\eta}^{n+1}), (p^{n+1},\bm{g}^{n+1})\right) \in Y \times Z$. 
\end{theorem}
Having shown the semi-discrete system is well-posed, we next consider the well-posedness of the fully discrete saddle point system.

\section{Fully Discrete Model.}\label{sec:FullyDiscreteModel}

In this section, we consider the finite element discretization of the temporally discretized form \eqref{WeakFormTimeDisc}. Let $h_1, h_2, \text{and } h_\gamma$ represent the mesh sizes of $\Omega_f, \Omega_s, \text{and } \gamma$.
Denote by $U^{h_1} \subset U, Q^{h_1} \subset Q, X^{h_2} \subset X, \Lambda^{h_\gamma} \subset H^{1/2}(\gamma)$ the discrete finite element spaces for $\bm{u},\bm{p},\bm{\eta},\bm{g}$. For the proof of the discrete inf-sup condition, higher regularity for $\bm{g}$ is required. We assume $\Lambda^{h_\gamma}$ is a subspace of $H^{1/2}(\gamma)$ instead of $H^{-1/2}(\gamma)$ for the analysis of the discrete model.
Set $Y^h := U^{h_1} \times X^{h_2}$, $Z^h := Q^{h_1} \times \Lambda^{h_\gamma}$.
Posing \eqref{WeakFormTimeDisc} over the finite element spaces, and using the bilinear forms $a(\cdot,\cdot)$ and $ b(\cdot,\cdot)$ defined as in \eqref{eq:bilinearForms}, the fully discrete saddle point problem reads: \\
\textit{Find $\left((\bm{u}_h^{n+1}, \bm{\eta}_h^{n+1}), (p_h^{n+1},\bm{g}_h^{n+1})\right) \in Y^h \times Z^h$ such that}
\begin{align}\label{eq:SaddlePtDisc}
\begin{split}
a((\bm{u}_h^{n+1},\frac{1}{\Delta t}\bm{\eta}_h^{n+1}),(\bm{v}_h,\bm{\varphi}_h)) + b((\bm{v}_h,\bm{\varphi}_h),(\Delta t p_h^{n+1},\Delta t \bm{g}_h^{n+1})) &= \mathcal{F}_{h,1}(\bm{v}_h,\bm{\varphi}_h) \hspace{5mm} \forall \hspace{1mm} (\bm{v}_h, \bm{\varphi}_h) \in Y^h, \\
b((\bm{u}_h^{n+1},\frac{1}{\Delta t}\bm{\eta}_h^{n+1}),(q_h,\bm{s}_h)) &= \mathcal{F}_{h,2}(q_h,\bm{s}_h) \hspace{5mm} \forall \hspace{1mm} (q_h, \bm{s}_h) \in Z^h. 
\end{split}
\end{align}

As in the semi-discrete case, we may show the well-posedness of the saddle point system by proving the coercivity of $a(\cdot,\cdot)$ and the inf-sup condition between $Y^h$ and $Z^h$ for $b(\cdot,\cdot)$.
Since $a(\cdot;\cdot)$ was proven to be coercive in $Y$ in the previous section, and $Y^h \subset Y$, $a(\cdot;\cdot)$ is also coercive on $Y^h$. Thus, all that remains to show is the discrete inf-sup condition.

For examining the discrete model, we assume that $U^{h_1}, Q^{h_1}$ satisfy the discrete inf-sup condition for the traditional Stokes problem:
\begin{equation*}
\underset{0\neq q^h \in Q^{h_1}}{\inf} \underset{0\neq \bm{v}^h \in U^{h_1}}{\sup} \frac{(\nabla \cdot \bm{v}^h, q^h)}{||\bm{v}^h||_1||q^h||_0} \geq \beta^* > 0.
\end{equation*}
The following subspace of $U^{h_1}$ is needed for analysis: $$U^{h_1}_{\gamma 0} = \{\bm{u}^h \in U^{h_1} \hspace{1mm} \big| \hspace{1mm} \bm{u}^h\big|_\gamma = 0\}.$$
Since the discrete LM space $\Lambda^{h_\gamma}$ is a subspace of $H^{1/2}(\gamma)$, the following inverse inequality holds \cite{Temam_1984}:
\begin{align}\label{InvIneq}
||\bm{k}^h||_{s,\gamma} \leq C_2 h_{\gamma}^{t-s} ||\bm{k}^h||_{t,\gamma}, \hspace{5mm} \forall \medskip \bm{k}^h \in \Lambda^{h_\gamma}, \hspace{5mm} -\frac{1}{2} \leq t \leq s \leq \frac{1}{2}.
\end{align}
We also assume standard approximation properties for the finite element spaces; see, for example, \cite{Gunzburger_1992}.

\begin{theorem}\label{thm:DiscInfSup}
For $\frac{h_\gamma}{h_1}, \frac{h_\gamma}{h_2}$ sufficiently large, there exists a positive constant $\beta$ such that:
\begin{align}\label{infsup-disc}
\underset{0 \neq \bm{y}^h \in Y^h}{\sup} \frac{b(\bm{y}^h,\bm{z}^h)}{||\bm{y}^h||_{Y}} \geq \beta ||\bm{z}^h||_{Z} > 0 \hspace{3mm} \forall \bm{z}^h \in Z^h.
\end{align}
\end{theorem}

\begin{proof}
    Let $\bm{z}^h := (q^h, \bm{s}^h) \in Z^h$ be given.
    
Let $(\hat{\bm{u}_1}, \hat{p}_1) \in U \times Q$ be the solution of the following continuous Stokes problem:
\begin{align}\label{SP1}
\begin{split}
(\hat{\bm{u}_1}, \bm{v})_{\Omega_f} + (D(\hat{\bm{u}_1}), D(\bm{v}))_{\Omega_f} + (\nabla \cdot \bm{v},\hat{p}_1) &= -(\bm{s}^h, \bm{v})_\gamma \hspace{5mm} \forall \hspace{1mm} \bm{v} \in U, \\
(\nabla \cdot \hat{\bm{u}_1},r) &= 0 \hspace{5mm} \forall \hspace{1mm} r \in Q.
\end{split}
\end{align}
The existence and uniqueness of $(\hat{\bm{u}_1}, \hat{p}_1)$ come from the inf-sup condition of $(\nabla \cdot \bm{v},q)$ on $U \times Q$, as well as the coercivity of the terms  $(\hat{\bm{u}_1}, \bm{v})_{\Omega_f} + (D(\hat{\bm{u}_1}), D(\bm{v}))_{\Omega_f}$. 
Then, we can bound the norms of the solution by the norm of the data: 
\begin{align*}
||\hat{\bm{u}_1}||_1 + ||\hat{p}_1||_0 \leq C ||\bm{s}^h||_{-1/2,\gamma}.
\end{align*}

\noindent Thus, by regularity theories for the Stokes equations (see \cite{Temam_1984}) and by \eqref{InvIneq}, we get
\begin{align}
||\hat{\bm{u}}_1||_2 + ||\hat{p}||_1 \leq C_1 ||\bm{s}^h||_{1/2,\gamma} \leq \frac{C_1 C_2}{h_\gamma} ||\bm{s}^h||_{-1/2,\gamma}.
\end{align} 

\noindent Likewise, let $(\hat{\bm{u}}^h_1, \hat{p}^h_1) \in U^{h_1} \times Q^{h_1}$ be the solution of the corresponding discrete Stokes problem:

\begin{align}\label{SP2}
\begin{split}
(\hat{\bm{u}}_1^h, \bm{v}^h)_{\Omega_f} + (D(\hat{\bm{u}}_1^h), D(\bm{v}^h))_{\Omega_f} + (\nabla \cdot \bm{v}^h,\hat{p}_1^h) &= -(\bm{s}^h, \bm{v}^h)_\gamma \hspace{5mm} \forall \hspace{1mm} \bm{v}^h \in U^{h_1}, \\
(\nabla \cdot \hat{\bm{u}}_1^h,r^h) &= 0 \hspace{5mm} \forall \hspace{1mm} r^h \in Q^{h_1}.
\end{split}
\end{align}
Again, the existence and uniqueness of $(\hat{\bm{u}}_1^h, \hat{p}_1^h)$ come from the inf-sup condition of $(\nabla \cdot \bm{v}^h,q)$ on $U^{h_1} \times Q^{h_1}$.
Using approximation properties of the discrete spaces, we can derive the following error estimate:
\begin{align}\label{errorEst}
\begin{split}
||\hat{\bm{u}}_1 - \hat{\bm{u}}_1^h||_1 \leq C(\underset{\bm{v}^h \in U^{h_1}}{\inf} ||\hat{\bm{u}}_1 - \bm{v}^h||_1 + \underset{q^h \in Q^{h_1}}{\inf} ||\hat{p}_1 - q^h||_0) \leq
C_3 h_1 ( ||\hat{\bm{u}_1}||_2 + ||\hat{p}_1||_1) \leq C_1 C_2 C_3 \frac{h_1}{h_\gamma} ||\bm{s}^h||_{-1/2,\gamma}. 
\end{split}
\end{align}

Define the space $\tilde{K} := \{\bm{v} \in U: b_p(\bm{v},r) = 0, \hspace{2mm} \forall r \in L^2(\Omega_f)\}$. Applying Lemma \ref{lemma:bpExist} with $\tilde{q}=0$, we may find a $\bm{v} \in \tilde{K}$ such that for a given $\bm{k} \in H^{1/2}(\gamma)$,
\begin{equation}\label{eq:vTok}
 \bm{v}\big|_\gamma = \bm{k}, \quad
||\bm{v}||_1 \leq C_4 ||\bm{k}||_{1/2,\gamma}. 
\end{equation}
 With \eqref{SP1} and \eqref{eq:vTok}, we can derive the following bounds:
\begin{align}\label{UboundSNorm}
\begin{split}
||\bm{s}^h||_{-1/2,\gamma} &= \underset{\bm{k} \in H^{1/2}(\gamma)}{\sup} \frac{(\bm{s}^h, \bm{k})_\gamma}{||\bm{k}||_{1/2,\gamma}} 
 \leq \underset{\bm{k} \in H^{1/2}(\gamma)}{\sup} C_4 \frac{(\bm{s}^h, \bm{k})_\gamma}{||\bm{v}||_{1}}  
 \leq \underset{\bm{v} \in \tilde{K}}{\sup} \hspace{2mm} C_4 \frac{(\bm{s}^h, \bm{v})_\gamma}{||\bm{v}||_{1}}   \\
&= C_4 \hspace{2mm} \underset{\bm{v} \in \tilde{K}}{\sup} \frac{-(\hat{\bm{u}_1}, \bm{v})_{\Omega_f} - (D(\hat{\bm{u}_1}), D(\bm{v}))_{\Omega_f} - (\nabla \cdot \bm{v},\hat{p}_1)}{||\bm{v}||_1} \\
&\leq \overline{C_4} ||\hat{\bm{u}}_1||_1.
\end{split}
\end{align}
Using the bounds in \eqref{errorEst} and \eqref{UboundSNorm}, we see that for $\frac{h_\gamma}{h_1} > K_1 := C_1 C_2 C_3 \overline{C_4}$,
\begin{align}\label{ubiggers}
\begin{split}
||\hat{\bm{u}}_1^h||_1 &\geq ||\hat{\bm{u}}_1||_1 - ||\hat{\bm{u}}_1 - \hat{\bm{u}}_1^h||_1  \\
& \geq \frac{1}{\overline{C_4}} ||\bm{s}^h||_{-1/2,\gamma} - C_1 C_2 C_3 \frac{h_1}{h_\gamma} ||\bm{s}^h||_{-1/2,\gamma} \hspace{10mm} \\ 
& \geq C ||\bm{s}^h||_{-1/2,\gamma}, 
\end{split}
\end{align}
From \eqref{SP2}, we derive an expression for $||\hat{\bm{u}}_1^h||_1$, letting $\bm{v}^h = \hat{\bm{u}}_1^h$ and $r^h=\hat{p}_1^h$: 
\begin{align}\label{Ubound}
\begin{split}
(\hat{\bm{u}}_1^h, \hat{\bm{u}}_1^h)_{\Omega_f} + (D(\hat{\bm{u}}_1^h), D(\hat{\bm{u}}_1^h))_{\Omega_f}  &= -(\bm{s}^h, \hat{\bm{u}}_1^h)_\gamma,  \\
||\hat{\bm{u}}_1^h||_1^2 &= -(\bm{s}^h, \hat{\bm{u}}_1^h)_\gamma.
\end{split}
\end{align}

\noindent Now, we define $\hat{\bm{\eta}} \in X$ as the solution of the following problem:
\begin{align}\label{SP3}
(\hat{\bm{\eta}}, \bm{\varphi})_{\Omega_s} + (D(\hat{\bm{\eta}}), D(\bm{\varphi}))_{\Omega_s} + (\nabla \cdot \hat{\bm{\eta}}, \nabla \cdot \bm{\varphi})_{\Omega_s} &= (\bm{s}^h, \bm{\varphi})_\gamma \hspace{5mm} \forall \hspace{1mm} \bm{\varphi} \in X. 
\end{align}
Likewise, let $\hat{\bm{\eta}}^h \in X^{h_2}$ be the solution to the discretized version of \eqref{SP3}.
\begin{align}\label{SP4}
(\hat{\bm{\eta}}^h, \bm{\varphi}^h)_{\Omega_s} + (D(\hat{\bm{\eta}}^h), D(\bm{\varphi}^h))_{\Omega_s} + (\nabla \cdot \hat{\bm{\eta}}^h, \nabla \cdot \bm{\varphi}^h)_{\Omega_s} &= (\bm{s}^h, \bm{\varphi}^h)_\gamma \hspace{5mm} \forall \hspace{1mm} \bm{\varphi}^h \in X^{h_2}. 
\end{align}
Again by regularity theories and by \eqref{InvIneq}, we can write
\begin{align}
||\hat{\bm{\eta}}||_2 \leq C_6||\bm{s}^h||_{1/2,\gamma} \leq \frac{C_6 C_2}{h_\gamma} ||\bm{s}^h||_{-1/2,\gamma}.
\end{align}
We use this bound, along with corresponding approximation properties, to derive the error estimate
\begin{align}\label{EtaError}
||\hat{\bm{\eta}} - \hat{\bm{\eta}}^h||_1 \leq C \underset{\bm{\varphi}^h \in X^{h_2}}{\inf} ||\hat{\bm{\eta}} - \bm{\varphi}^h||_1 \leq C_7 h_2 ||\hat{\bm{\eta}}||_2 \leq C_7 C_6 C_2\frac{h_2}{h_\gamma} ||\bm{s}^h||_{-1/2,\gamma}.
\end{align}

\noindent For any $\bm{k} \in H^{1/2}(\gamma)$, we can choose a $\bm{\varphi} \in X$ such that 
\begin{equation}\label{phitok}
\bm{\varphi}\Big|_\gamma = \bm{k}, \quad \text {and} \quad
||\bm{\varphi}||_1 \leq C_5||\bm{k}||_{1/2,\gamma}.
\end{equation}
As before, we can write an expression for the norm of $\bm{s}^h$ using \eqref{SP3} and \eqref{phitok}:
\begin{align}\label{EtaSupBound}
\begin{split}
||\bm{s}^h||_{-1/2,\gamma} &= \underset{\bm{k} \in H^{1/2}(\gamma)}{\sup} \frac{(\bm{s}^h, \bm{k})_\gamma}{||\bm{k}||_{1/2,\gamma}}  \leq \underset{\bm{k} \in H^{1/2}(\gamma)}{\sup} C_5 \frac{(\bm{s}^h, \bm{k})_\gamma}{||\bm{\varphi}||_{1}}  
\leq \underset{\bm{\varphi} \in X}{\sup} \hspace{2mm} C_5 \frac{(\bm{s}^h, \bm{\varphi})_\gamma}{||\bm{\varphi}||_{1}}   \\
&= C_5 \hspace{2mm} \underset{\bm{\varphi} \in X}{\sup} \frac{(\hat{\bm{\eta}}, \bm{\varphi})_{\Omega_s} + (D(\hat{\bm{\eta}}), D(\bm{\varphi}))_{\Omega_s} + (\nabla \cdot \hat{\bm{\eta}},\nabla \cdot \bm{\varphi})_{\Omega_s}}{||\bm{\varphi}||_1} \\
&\leq \overline{C_5} ||\hat{\bm{\eta}}||_1.
\end{split}
\end{align}
Using the bounds in \eqref{EtaError} and \eqref{EtaSupBound}, we see that for $\frac{h_\gamma}{h_2} > K_2 := \overline{C_5} C_6 C_7 C_2$,
\begin{align}\label{etaToUse}
\begin{split}
||\hat{\bm{\eta}}^h||_1 &\geq ||\hat{\bm{\eta}}||_1 - ||\hat{\bm{\eta}} - \hat{\bm{\eta}}^h||_1 \\
&\geq \frac{1}{\overline{C_5}} ||\bm{s}^h||_{-1/2,\gamma}  - C_7 C_6 C_2 \frac{h_2}{h_\gamma} ||\bm{s}^h||_{-1/2,\gamma} \\
&\geq C ||\bm{s}^h||_{-1/2,\gamma}, 
\end{split}
\end{align}
Now, letting $\bm{\varphi}^h = \hat{\bm{\eta}}^h$ in \eqref{SP4},
\begin{align}\label{EtaBound}
\begin{split}
(\hat{\bm{\eta}}^h, \hat{\bm{\eta}}^h)_{\Omega_s} + (D(\hat{\bm{\eta}}^h), D(\hat{\bm{\eta}}^h))_{\Omega_s} + (\nabla \cdot \hat{\bm{\eta}}^h, \nabla \cdot \hat{\bm{\eta}}^h)_{\Omega_s} &= (\bm{s}^h, \hat{\bm{\eta}}^h)_\gamma \\
||\hat{\bm{\eta}}^h||_1^2 + || \nabla \cdot \hat{\bm{\eta}}^h||_0^2 &= (\bm{s}^h, \hat{\bm{\eta}}^h)_\gamma \\
||\hat{\bm{\eta}}^h||_1^2  &\leq (\bm{s}^h, \hat{\bm{\eta}}^h)_\gamma. \\
\end{split}
\end{align}
Lastly, since we know the inf-sup condition holds on $U^{h_1}_{\gamma 0} \times Q^{h_1}$, we choose $\hat{\bm{u}}_2^h \in U^{h_1}_{\gamma 0}$ such that
\begin{align}
(\nabla \cdot \hat{\bm{u}}_2^h,r^h) &= -(q^h, r^h) \hspace{5mm} \forall r^h \in Q^{h_1}, \label{U2Def}\\
||\hat{\bm{u}}_2^h||_1 &\leq C ||q^h||_0. \label{U2Norm}
\end{align}

Now, set $\hat{\bm{u}}^h := \hat{\bm{u}}_1^h + \hat{\bm{u}}_2^h$. Then for the given $(q^h, \bm{s}^h) \in Z^h$,
\begin{align*}
b(\hat{\bm{u}}^h, \hat{\bm{\eta}}^h; q^h, \bm{s}^h) &= (\hat{\bm{\eta}}^h, \bm{s}^h)_\gamma -  (\hat{\bm{u}}^h, \bm{s}^h)_\gamma -  (\nabla \cdot \hat{\bm{u}}^h; q^h) \\
&= (\hat{\bm{\eta}}^h, \bm{s}^h)_\gamma -  (\hat{\bm{u}}_1^h, \bm{s}^h)_\gamma -  (\hat{\bm{u}}_2^h, \bm{s}^h)_\gamma - (\nabla \cdot \hat{\bm{u}}^h_1; q^h) -  (\nabla \cdot \hat{\bm{u}}^h_2; q^h) \\
&= (\hat{\bm{\eta}}^h, \bm{s}^h)_\gamma +  ||\hat{\bm{u}}_1^h||_1^2 + ||q^h||_0^2 \hspace{10mm} \text{ by \eqref{Ubound}, \eqref{SP2}, and \eqref{U2Def}} \\
& \geq ||\hat{\bm{\eta}}^h||_1^2 +  ||\hat{\bm{u}}_1^h||_1^2 + ||q^h||_0^2 \hspace{15mm} \text{ by \eqref{EtaBound}} \\
& \geq \frac{1}{2} ||\hat{\bm{\eta}}^h||_1^2 + \frac{1}{2} ||q^h||_0^2  + \frac{1}{2} ||\hat{\bm{u}}_1^h||_1^2 + \frac{1}{2} ||q^h||_0^2 \\
& \geq \frac{1}{4} \Big(||\hat{\bm{\eta}}^h||_1 + ||q^h||_0\Big)\Big(||\hat{\bm{\eta}}^h||_1 + ||q^h||_0\Big) + \frac{1}{4} \Big(||\hat{\bm{u}}_1^h||_1 + ||q^h||_0\Big)\Big(||\hat{\bm{u}}_1^h||_1 + ||q^h||_0\Big)  \\
& \geq \frac{1}{4}C \Big[ \Big(||\bm{s}^h||_{-1/2,\gamma} + ||q^h||_0\Big)\Big(||\hat{\bm{\eta}}^h||_1 + ||q^h||_0\Big) + \Big(||\bm{s}^h||_{-1/2,\gamma} + ||q^h||_0\Big)\Big(||\hat{\bm{u}}_1^h||_1 + ||\hat{\bm{u}}_2^h||_1\Big)\Big]\\ &\hspace{95mm} \text{ by \eqref{etaToUse}, \eqref{ubiggers}, \eqref{U2Norm}} \\
&\geq \frac{1}{4} C \Big(||\bm{s}^h||_{-1/2,\gamma} + ||q^h||_0\Big) \Big(||\hat{\bm{\eta}}^h||_1 + ||q^h||_0 + ||\hat{\bm{u}}^h||_1\Big) \\
& \geq \frac{1}{4} C \Big(||\bm{s}^h||_{-1/2,\gamma} + ||q^h||_0\Big) \Big(||\hat{\bm{\eta}}^h||_1 + ||\hat{\bm{u}}^h||_1\Big) \\
&\geq \frac{1}{4} C \Big(||\bm{s}^h||_{-1/2,\gamma}^2 + ||q^h||_0^2\Big)^{1/2} \Big(||\hat{\bm{\eta}}^h||_1^2 + ||\hat{\bm{u}}^h||_1^2\Big)^{1/2}  \\ 
&= \frac{1}{4} C ||q^h; \bm{s}^h||_Z \hspace{1mm} ||\hat{\bm{u}}^h; \hat{\bm{\eta}}^h||_Y.
\end{align*}
Finally, 
\begin{align*}
\frac{b(\hat{\bm{u}}^h, \hat{\bm{\eta}}^h; q^h, \bm{s}^h)}{||\hat{\bm{u}}^h; \hat{\bm{\eta}}^h||_Y} \geq \beta^* ||q^h; \bm{s}^h||_Z. 
\end{align*}
Since $(q^h, \bm{s}^h) \in Z^h$ was arbitrary, taking the supremum over $0 \neq (\bm{u}^h, \bm{\eta}^h) \in Y^h$, we have the discrete inf-sup condition \eqref{infsup-disc}.
\end{proof}

\noindent Using Theorem \ref{thm:DiscInfSup} and the coercivity of $a(\cdot, \cdot)$ on $Y^h$, we obtain the well-posedness of \eqref{eq:SaddlePtDisc}.
\begin{theorem} \label{fully-discrete-well-posedness}
The fully discrete system \eqref{eq:SaddlePtDisc} has a unique solution $\Big((\bm{u}_h^{n+1}, \bm{\eta}_h^{n+1}), (p_h^{n+1},\bm{g}_h^{n+1})\Big) \in Y^h \times Z^h$.
\end{theorem}
Since the fully discrete saddle point system is well-posed, we may now examine a solution method for the matrix system arising from the fully discrete formulation.

\section{Partitioned Method.}\label{sec:IFRsec}

In this section, we present the fully discretized formulation and our partitioned scheme for its solution. Let $\{v_j\}, \{q_j\}, \{\phi_j\},$ and $\{\mu_j\}$ be basis functions for the discrete solutions $\bm{u}^h, p^h, \bm{\eta}^h,$ and $\bm{g}^h$ of \eqref{WeakForm} and $\bm{u},\bm{p},\bm{\eta},\bm{g}$ represent the corresponding coefficient vectors. Then the discrete solutions may be expressed as linear combinations of these basis functions, i.e., $\bm{u}^h(x,y,t) = \sum_j \bm{u}_j(t) v_j(x,y)$. Substituting these linear combinations into \eqref{WeakFormTimeDisc} and taking the test functions to be the appropriate basis functions results in a linear system:
\begin{align}\label{SchurDiscTime}
\begin{split}
 M_f \bm{u}^{n+1} +  \Delta t K_f \bm{u}^{n+1} - \Delta t P \bm{p}^{n+1} - \Delta t G_f^T \bm{g}^{n+1} &= \Delta t \overline{\bm{f_f}}^{n+1} + M_f \bm{u}^n, \\
P^T \bm{u}^{n+1} &= \bm{0}, \\
\frac{1}{\Delta t} M_s \bm{\eta}^{n+1} + \Delta t (K_s + L) \bm{\eta}^{n+1} + \Delta t G_s^T \bm{g}^{n+1} &= \Delta t \overline{\bm{f_s}}^{n+1} + \frac{2}{\Delta t} M_s \bm{\eta}^n - \frac{1}{\Delta t} M_s \bm{\eta}^{n-1},\\
 \frac{1}{\Delta t} G_s \bm{\eta}^{n+1} -  G_f \bm{u}^{n+1} &=  \frac{1}{\Delta t} G_s \bm{\eta}^n,
\end{split}
\end{align}
where $M_f, K_f, M_s, K_s$ are mass and stiffness matrices for the fluid and structure, $P$ is the pressure matrix, $L$ represents the divergence terms for the structure, and $\overline{\bm{f}_f}, \overline{\bm{f}_s}$ contain the body forces and Neumann boundary terms. The $G$ matrices represent the interaction between the interface and subdomain bases: for $d=2$,  
\begin{align*}
G_f = \begin{bmatrix}
G_{f}^1 & 0 \\ 0 & G_{f}^2
\end{bmatrix}, \hspace{5mm} \text{ with } &\Big(G_{f}^r\Big)_{i,j} = \langle v_j, \mu_i\rangle_\gamma \text{  for } r=1,2 \\
G_s = \begin{bmatrix}
G_{s}^1 & 0 \\ 0 & G_{s}^2
\end{bmatrix}, \hspace{5mm} \text{ with } &\Big(G_{s}^r\Big)_{i,j} = \langle \varphi_j, \mu_i\rangle_\gamma \text{  for } r=1,2. 
\end{align*}
Let $N_u, N_p, N_\eta, N_\gamma $ be the total number of degrees of freedom for each variable. Matrix dimensions are
$$M_f, K_f: N_u \times N_u \hspace{5mm} 
M_s, K_s, L_s: N_\eta \times N_\eta \hspace{5mm} P: N_u \times N_p \hspace{5mm} G_f: N_\gamma \times N_u \hspace{5mm} G_s:N_\gamma \times N_\eta.$$ 

To solve the original FSI system, \eqref{StokesMom}-\eqref{Inter2}, we now must solve the matrix system \eqref{SchurDiscTime}. Our goal is to solve this system by expressing the LM coefficient vector $\bm{g}^{n+1}$ as an implicit function of the other variables, decoupling the system. 
At this stage, we can solve the fluid and structure equations for $\bm{g}^{n+1}$ and substitute into the constraint equation to express $\bm{g}^{n+1}$ in terms of all three physical variables. However, without initial conditions for the pressure variable, or a way to separately update it on each time step, we will be unable to fully decouple the system. Thus, we consider grouping the pressure variable and Lagrange multiplier in order to effect the desired decoupling. We define  
\begin{align*}
\bm{z}^{n+1} := \begin{bmatrix}
\bm{p}^{n+1}  \\ \bm{g}^{n+1}
\end{bmatrix}, \hspace{5mm} A_f := \begin{bmatrix}
P^T \\ G_f
\end{bmatrix}, \hspace{5mm} A_s := \begin{bmatrix}
0_{N_p \times N_\eta} \\ G_s
\end{bmatrix}.
\end{align*}
For consistency with the variables used in the saddle point system \eqref{eq:SaddlePtDisc}, we consider the variables $\tilde{\bm{\eta}}^{n+1} := \frac{1}{\Delta t} \bm{\eta}^{n+1}$ and $\tilde{\bm{z}}^{n+1} := \Delta t \bm{z}^{n+1}$.
Using these definitions, we rewrite \eqref{SchurDiscTime}:
\begin{align}\label{withA}
\begin{split}
M_f \bm{u}^{n+1} + \Delta t K_f \bm{u}^{n+1} -  A_f^T \tilde{\bm{z}}^{n+1} &= \Delta t \overline{\bm{f_f}}^{n+1} +  M_f \bm{u}^n, \\
 M_s \tilde{\bm{\eta}}^{n+1} + \Delta t^2 (K_s + L) \tilde{\bm{\eta}}^{n+1} +  A_s^T \tilde{\bm{z}}^{n+1} &= \Delta t \overline{\bm{f_s}}^{n+1} + \frac{2}{\Delta t} M_s \bm{\eta}^n - \frac{1}{\Delta t} M_s \bm{\eta}^{n-1},\\
 A_s \tilde{\bm{\eta}}^{n+1} -  A_f \bm{u}^{n+1} &= \begin{bmatrix}
0_{N_p \times 1} \\  \frac{1}{\Delta t} G_s \bm{\eta}^n
\end{bmatrix}.
\end{split}
\end{align}

\noindent For ease of notation, define
\begin{align*}
W_f &:=M_f + \Delta t K_f,  \hspace{20mm}
W_s :=  M_s + \Delta t^2 (K_s + L), \\
\bm{w_1}^n &:= \Delta t\overline{\bm{f_f}}^{n+1} +  M_f \bm{u}^n, \hspace{10mm}
\bm{w_2}^n :=  \Delta t \overline{\bm{f_s}}^{n+1} + \frac{2}{\Delta t} M_s \bm{\eta}^n - \frac{1}{\Delta t} M_s \bm{\eta}^{n-1}, \hspace{10mm}
\bm{w_3}^n &:= \begin{bmatrix}
0_{N_p \times 1} \\ \frac{1}{\Delta t} G_s \bm{\eta}^n
\end{bmatrix}.
\end{align*}
Assuming that $W_f, W_s$ are full rank, we can solve the first two equations of \eqref{withA} for $\bm{u}^{n+1}$ and $\bm{\eta}^{n+1}$:
\begin{align}\label{aandc}
\begin{split}
\bm{u}^{n+1} &= W_f^{-1} (\bm{w_1}^n +  A_f^T \tilde{\bm{z}}^{n+1}), \\
\tilde{\bm{\eta}}^{n+1} &= W_s^{-1} (\bm{w_2}^n -  A_s^T \tilde{\bm{z}}^{n+1}) .
\end{split}
\end{align} 
Plugging these into the third equation of \eqref{withA} and rearranging yields:
\begin{align}\label{SchurBE}
 \Big( A_f W_f^{-1} A_f^T  + A_s W_s^{-1} A_s^T \Big) \tilde{\bm{z}}^{n+1} = A_s W_s^{-1} \bm{w_2}^n - A_f W_f^{-1} \bm{w_1}^n - \bm{w_3}^n.
\end{align}

\noindent We refer to \eqref{SchurBE} as the Schur complement equation, with matrix $S :=  A_f W_f^{-1} A_f^T  +  A_s W_s^{-1} A_s^T  $. As desired, it expresses the variable $\tilde{\bm{z}}^{n+1}$ implicitly in terms of the variables $\bm{u}^n, \bm{\eta}^n, \bm{\eta}^{n-1}$. If $S$ is full rank, the Schur complement equation can be solved for $\tilde{\bm{z}}^{n+1}$, which will allow us to solve for  $\bm{u}^{n+1}$ and $\tilde{\bm{\eta}}^{n+1}$ in \eqref{aandc} explicitly and independently. The partitioned method is summarized in Algorithm \ref{alg:SchurAlg}.


\begin{algorithm}
\caption{Schur Complement Algorithm}
\label{alg:SchurAlg}
\begin{algorithmic}
\STATE{Let $\tilde{\bm{\eta}}^{n+1} := \frac{1}{\Delta t} \bm{\eta}^{n+1}$ and $\tilde{\bm{z}}^{n+1} := \Delta t \bm{z}^{n+1}$. Given $\bm{u}^0, \bm{\eta}^0, \bm{\dot{\eta}}^0$ and $N = T / \Delta t$:}
\FOR{n=0,1,...N-1}
\STATE Compute
$\bm{w}_1^n :=  \Delta t \overline{\bm{f_f}}^{n+1} + M_f \bm{u}^n \text{ and }
\bm{w}_3^n :=  \begin{bmatrix}
0_{N_p \times 1} \\ \frac{1}{\Delta t}G_s \bm{\eta}^n
\end{bmatrix}$
\IF{n=0} 
\STATE $\bm{w}_2^0 := \Delta t \overline{\bm{f_s}}^{1} + \frac{1}{\Delta t} M_s \bm{\eta}^0 +  M_s \bm{\dot{\eta}}^0$ \\
\ELSE
\STATE
$ \bm{w}_2^n := \Delta t \overline{\bm{f_s}}^{n+1} + \frac{2}{\Delta t}M_s \bm{\eta}^n - \frac{1}{\Delta t} M_s \bm{\eta}^{n-1}
$
\ENDIF
\STATE Solve 
$\left( A_f W_f^{-1} A_f^T  + A_s W_s^{-1} A_s^T \right) \tilde{\bm{z}}^{n+1} =  A_s W_s^{-1} \bm{w}_2^n -  A_f W_f^{-1} \bm{w}_1^n - \bm{w_3}^n 
$  
\STATE Solve
$
W_f \bm{u}^{n+1} = \bm{w}_1^n + A_f^T \tilde{\bm{z}}^{n+1} \text{ and } 
W_s \tilde{\bm{\eta}}^{n+1} = \bm{w}_2^n - A_s^T \tilde{\bm{z}}^{n+1} $
\ENDFOR
\RETURN $\bm{u}^{n+1}, \tilde{\bm{\eta}}^{n+1}$.
\end{algorithmic}
\end{algorithm}
\begin{remark}
    The pressure can be extracted at any time from the variable $\bm{z}^{n+1} = \frac{1}{\Delta t} \tilde{\bm{z}}^{n+1}$. We recover $\bm{\eta}^{n+1} = \Delta t \tilde{\bm{\eta}}^{n+1}$.
\end{remark}

\subsection{Preconditioner}
Consider the Schur complement equation \eqref{SchurBE}, solved in the second stage of Algorithm \ref{alg:SchurAlg}. The Schur complement matrix, $S$, is time-independent, hence, the use of a direct solver for \eqref{SchurBE} can be considered. The system is also solvable without explicitly constructing $S =  A_f W_f^{-1} A_f^T +  A_s W_s^{-1} A_s^T$. Since $S$ is symmetric and positive definite, we may use the conjugate gradient algorithm to solve \eqref{SchurBE}, which only requires matrix-vector products. If the product $Sz$ is desired, it can be constructed as follows:
\begin{enumerate}
\item Solve the system $W_f x_f = A_f^T z$ for $x_f$. 
\item Solve the system $W_s x_s = A_s^T z$ for $x_s$.
\item Then $Sz =  A_f x_f +  A_s x_s$.
\end{enumerate}
Also, note that in the right hand side of \eqref{SchurBE}, the inverse matrices $W_s^{-1}$ and $W_f^{-1}$ appear. We may also avoid constructing these by solving the systems $W_s \bm{y_s} = \bm{w}_2^n$ and $W_f \bm{y_f} = \bm{w}_1^n$. Then the right hand side is calculated as:
\begin{align*}
F^n &=  A_s W_s^{-1} \tilde{\bm{w}_2}^n -  A_f W_f^{-1} \tilde{\bm{w}_1}^n - \bm{w_3}^n =  A_s \bm{y_s} - A_f \bm{y_f} - \bm{w_3}^n.
\end{align*}

Note that we can consider the Schur complement matrix as the sum of subdomain matrices; i.e. $S = S_f + S_s$, where $S_f =  A_f W_f^{-1} A_f^T$ and $S_s =  A_s W_s^{-1} A_s^T$. If we want to use the preconditioned conjugate gradient (PCG) method to solve \eqref{SchurBE}, one option for a preconditioner for $S$ is to use either $S_f$ or $S_s$ \cite{Quarteroni_1999}. Using $S_f$ as the preconditioner gives the system $S_f^{-1} S \tilde{\bm{z}}^{n+1} = S_f^{-1} F^n$. Again, we do not need to construct $S_f$ or its inverse explicitly in the PCG algorithm. If we want the matrix-vector product $S_f^{-1} \bm{y}$, we can solve the following system:
\begin{align*}
\begin{bmatrix}
W_f & A_f^T \\ A_f & 0
\end{bmatrix} \begin{bmatrix}
\bm{\beta} \\ \bm{x}
\end{bmatrix} = \begin{bmatrix}
0 \\ \bm{y}
\end{bmatrix}
\end{align*}
for the vector $[\bm{\beta}, \bm{x}]^T$. Then $\bm{x} = -S_f^{-1}\bm{y}$.

\section{Conditioning of Schur Complement Matrix.}\label{sec:Conditioning}

Our approach to solving the FSI problem involves solution of a linear system involving the Schur complement matrix; thus, the conditioning of $S$ plays an important role in the performance of this algorithm. 

Recall the variational form for the singular values of a general matrix $B \in \mathbb{R}^{m\times n}$
\begin{align}\label{eq:variationalSigma}
    \sigma_{max}(B) = \underset{\bm{v} \in \mathbb{R}^n}{\text{max}} \underset{\bm{u} \in \mathbb{R}^m}{\text{max}} \frac{\bm{u}^T B \bm{v}}{|\bm{u}| |\bm{v}|}, \hspace{5mm}
     \sigma_{min}(B) = \underset{\bm{v} \in \mathbb{R}^n}{\text{min}} \underset{\bm{u} \in \mathbb{R}^m}{\text{max}} \frac{\bm{u}^T B \bm{v}}{|\bm{u}| |\bm{v}|},
\end{align}
with $\sigma_{max},\sigma_{min}$ being the largest and smallest singular values of the matrix.
Here, $|\cdot|$ represents the standard Euclidean norm for vectors. Likewise, the matrix norm $|B|$ induced by $|\cdot|$ is equivalent to the largest singular value, giving
\begin{equation}\label{eq:smaxNorm}
    |B| = \sigma_{max}(B), \text{ and } \hspace{1mm} |B^{-1}| = \frac{1}{\sigma_{min}(B)}.
\end{equation}
The condition number of $B$ may thus be written
\begin{align*}
    \kappa(B) &= |B| |B^{-1}| = \frac{\sigma_{max}(B)}{\sigma_{min}(B)}.
\end{align*}

\noindent From \eqref{SchurBE}, the Schur complement matrix is expressed as $S = A_s W_s^{-1} A_s^T + A_f W_f^{-1} A_f^T , $ which can be written in block form as follows: 
 \begin{align*}
 S &=  A W^{-1} A^T, \hspace{3mm} \text{ where } \\
        A = \begin{bmatrix}
            A_f \Big| A_s
        \end{bmatrix} = &\begin{bmatrix}
            P^T & 0 \\ G_f &  G_s
        \end{bmatrix}   \hspace{10mm}
        W = \begin{bmatrix}
            W_f & 0 \\ 0 & W_s
        \end{bmatrix}.
    \end{align*}
We would like to bound $\sigma_{max}(S)$ from above and bound $\sigma_{min}(S)$ from below. 
Using \eqref{eq:smaxNorm},
\begin{align*}
\sigma_{max}(S) &=  |S| \leq   |A| |W^{-1}| |A^T| 
= \sigma^2_{max}(A) \frac{1}{\mu_{min}(W)}.
\end{align*}
Since W is symmetric and positive definite, its eigenvalues ($\mu_i > 0$) and singular values ($\sigma_i$) are equivalent. Also, $W$ has an eigendecomposition; there exists a real orthogonal matrix $B$ such that $B^TW B = D$ and $B^T B = I$, where the columns of $B$ are the eigenvectors of $W$ and $D$ is a diagonal matrix containing the eigenvalues of $W$. All entries of $D$ must be positive. 

We now consider a bound for $\sigma_{\min}(S)$. From \eqref{eq:variationalSigma}, we can write 
\begin{align}
\begin{split}\label{eq:sigMinPart1}
 \sigma_{\min}(S) &= \underset{\bm{s} \in \mathbb{R}^{N_p + N_\gamma}}{\text{min}} \underset{\bm{g} \in \mathbb{R}^{N_p + N_\gamma}}{\text{max}} \frac{\bm{g}^T S \bm{s}}{|\bm{g}| |\bm{s}|} =   \underset{\bm{s} \in \mathbb{R}^{N_p + N_\gamma}}{\text{min}} \underset{\bm{g} \in \mathbb{R}^{N_p + N_\gamma}}{\text{max}} \frac{\bm{g}^T A W^{-1} A^T \bm{s}}{|\bm{g}| |\bm{s}|} \\
    & =   \underset{\bm{s} \in \mathbb{R}^{N_p + N_\gamma}}{\text{min}} \underset{\bm{g} \in \mathbb{R}^{N_p + N_\gamma}}{\text{max}} \frac{\bm{g}^T A B D^{-1}B^T A^T \bm{s}}{|\bm{g}| |\bm{s}|}.
    \end{split}
\end{align}
To examine this expression, define $R := AB$, where $R \in \mathbb{R}^{(N_p + N_\gamma) \times (N_u + N_\eta)}$. For simplicity, we say $R \in \mathbb{R}^{m\times n}$. Denote the $i^{th}$ column of $R$ by $\bm{r}_i$. We assume the entries of $D$ are ordered such that $\mu_{\max}(W) = d_{11} \geq \cdots \geq d_{nn} = \mu_{\min}(W) > 0.$
Then the min/max expression \eqref{eq:sigMinPart1} can be expanded as
\begin{align*}
   \underset{\bm{s}}{\min} \hspace{2mm} \underset{\bm{g}}{\max} \left( \bm{g}^T R D^{-1} R^T \bm{s} \right) & = \underset{\bm{s}}{\min} \hspace{2mm}\underset{\bm{g}} {\max} \left( d_{11}^{-1}(\bm{g}^T \bm{r}_1)(\bm{s}^T \bm{r}_1) + d_{22}^{-1}(\bm{g}^T \bm{r}_2)(\bm{s}^T \bm{r}_2) + \cdots + d_{nn}^{-1}(\bm{g}^T \bm{r}_n)(\bm{s}^T \bm{r}_n)       \right).
\end{align*}
For any given vector $\bm{s}$, we may find a $\bm{g}$ that makes each term positive. In particular, we always have the option to pick $\bm{g} = \bm{s}$. Thus, the $\bm{g}$ that maximizes the above expression must be one that makes each term positive. This allows us to bound each term from below, for any given $\bm{s}$:
 \begin{align*}
    &\underset{\bm{g}} {\max} \left( d_{11}^{-1}(\bm{g}^T \bm{r}_1)(\bm{s}^T \bm{r}_1) + d_{22}^{-1}(\bm{g}^T \bm{r}_2)(\bm{s}^T \bm{r}_2) + \cdots + d_{nn}^{-1}(\bm{g}^T \bm{r}_n)(\bm{s}^T \bm{r}_n)       \right) \\
    &\geq \underset{\bm{g}} {\max} \left( d_{11}^{-1}(\bm{g}^T \bm{r}_1)(\bm{s}^T \bm{r}_1) + d_{11}^{-1}(\bm{g}^T \bm{r}_2)(\bm{s}^T \bm{r}_2) + \cdots + d_{11}^{-1}(\bm{g}^T \bm{r}_n)(\bm{s}^T \bm{r}_n)       \right) 
\end{align*} 
Taking the minimum over $\bm{s}$ yields
\begin{align*}
    \sigma_{\min}(S) &=   \underset{\bm{s} \in \mathbb{R}^{N_p + N_\gamma}}{\text{min}} \underset{\bm{g} \in \mathbb{R}^{N_p + N_\gamma}}{\text{max}} \frac{\bm{g}^T R D^{-1} R^T \bm{s}}{|\bm{g}| |\bm{s}|} \\
    &\geq  d_{11}^{-1} \underset{\bm{s} \in \mathbb{R}^{N_p + N_\gamma}}{\text{min}} \underset{\bm{g} \in \mathbb{R}^{N_p + N_\gamma}}{\text{max}} \frac{\bm{g}^T A B B^T A^T \bm{s}}{|\bm{g}||\bm{s}|} \\
    &=  \mu_{\max}^{-1}(W) \underset{\bm{s} \in \mathbb{R}^{N_p + N_\gamma}}{\text{min}} \underset{\bm{g} \in \mathbb{R}^{N_p + N_\gamma}}{\text{max}} \frac{\bm{g}^T A A^T \bm{s}}{|\bm{g}||\bm{s}|} \\
    &=  \mu_{\max}^{-1}(W)\sigma_{\min}^2(A).
\end{align*}

\noindent Combining the expressions for $\sigma_{\max}(S)$ and $\sigma_{\min}(S)$, we write the condition number of the Schur complement matrix as
\begin{align}
    k(S) &= \frac{\sigma_{\max}(S)}{\sigma_{\min}(S)}  \leq \frac{ \mu^{-1}_{\min}(W) \sigma^{2}_{\max}(A)}{ \mu^{-1}_{\max}(W) \sigma^2_{\min}(A)} = (k(A))^2 k(W). \label{AW-cond}
\end{align}

The following lemmas are presented to estimate $k(A)$ and $k(W)$ in \eqref{AW-cond}. 
The first two lemmas are found in \cite{Quarteroni_1994}. See Section 6.3 in \cite{Quarteroni_1994} for proofs. 
\begin{lemma}\label{lem:normFEandvec}
 Assume $V^h$ is a generic finite element space with global shape functions $\{\phi_i\}$. Let $\mathbb{T}_h$ be a quasi-uniform family of triangulations of $\overline{\Omega}$. Then there exist positive constants $C_{\ell}, C_u$ such that for each $v^h \in V^h$, with $v^h = \sum_i v_i \phi_i$ and $\bm{v}$ the vector of the coefficients $v_i$,
\begin{align*}
    C_\ell h^d |\bm{v}|^2 \leq ||v^h||_{0}^2 \leq C_u h^d |\bm{v}|^2.
\end{align*}
\end{lemma}
\begin{lemma}\label{lem:stiffEigBound}
    The eigenvalues $\mu_i(K)$ of the stiffness matrix $K$ satisfy the following bounds:
    \begin{equation}\label{eq:stiffEigBound}
        \alpha \delta C_\ell h^d \leq \mu_i(K) \leq \alpha^* \delta C_u h^d (1+C^* h^{-2}).
    \end{equation}
    where $\alpha$ and $\alpha^*$ are the coercivity and continuity constants, respectively, and $\delta =2\nu_f$ or $2\nu_s$ for the FSI model.
\end{lemma}
Next, we present a similar result for the matrix $L$, comprised of the divergence terms. 
\begin{lemma}\label{lem:divEigBound}
    The eigenvalues $\mu_i(L)$ of the divergence matrix $L$ satisfy the following bounds:
    \begin{equation}\label{eq:divEigBound}
        0 \leq \mu_i(L) \leq \lambda d C_u C^* h^{d-2}.
    \end{equation}
\end{lemma}
\begin{proof}
See Appendix.
\end{proof}
We now have the following estimate for the condition number of $W$ using Lemmas \ref{lem:normFEandvec}-\ref{lem:divEigBound}.
\begin{lemma}\label{lem:condWBound}
The condition number of $W$ is estimated as the following:
\begin{align}\label{eq:condW}
k(W) &= \frac{\mu_{\max}(W)}{\mu_{\min}(W)} 
= \frac{\max\Bigl\{C_{\max,f}h_1^d(1+\Delta t h_1^{-2}), C_{\max,s}h_2^d\left(1 + \Delta t^2 h_2^{-2}\right)\Bigr\}}{ \min \Bigl\{C_{\min,f}h_1^d(1+\Delta t), C_{\min,s} h_2^d(1+\Delta t^2) \Bigr\}}.
\end{align}
\end{lemma}
\begin{proof}
See Appendix.
\end{proof}

\noindent Next, we examine the condition number of $
      A = \begin{bmatrix}
P^T & 0 \\ G_f & G_s
    \end{bmatrix}.$

\begin{lemma}\label{lemma:AsigmaBound}
For $\frac{h_\gamma}{h_1}, \frac{h_\gamma}{h_2}$ sufficiently large, the singular values of $A$ may be bounded as follows:
\begin{align*}
\sigma_{\max}(A) \leq C_{\max} h_2^{(d-2)/2} (h_2^d + h_\gamma^{d-1})^{1/2}, \hspace{3mm} \text{ and } \hspace{3mm} 
    \sigma_{\min}(A) \geq C_{\min} h_1^d.
    \end{align*}
    Thus, 
    \begin{align}\label{eq:condA}
        k(A) = \frac{\sigma_{\max}(A)}{\sigma_{\min}(A)} = \frac{C_{\max} h_2^{(d-2)/2} (h_2^d + h_\gamma^{d-1})^{1/2}}{C_{\min}h_1^d}.
    \end{align}
\end{lemma}
\begin{proof}
See Appendix.
\end{proof}

We are now ready to present the main results of this section.
Using \eqref{eq:condW} and  \eqref{eq:condA}, a bound for the condition number of $S$, the Schur complement matrix, is estimated as:
\begin{align*} 
k(S) &=   \left(\frac{C_{\max} h_2^{(d-2)/2} (h_2^d + h_\gamma^{d-1})^{1/2}}{C_{\min}h_1^d}\right)^2\frac{\max\Bigl\{C_{\max,f}h_1^d(1+\Delta t h_1^{-2}), C_{\max,s}h_2^d\left(1 + \Delta t^2 h_2^{-2}\right)\Bigr\}}{ \min \Bigl\{C_{\min,f}h_1^d(1+\Delta t), C_{\min,s} h_2^d(1+\Delta t^2) \Bigr\}} \\
&=\left(\frac{C_{\max} h_2^{(d-2)} (h_2^d + h_\gamma^{d-1})}{C_{\min}h_1^{2d}}\right)\frac{\max\Bigl\{C_{\max,f}h_1^d(1+\Delta t h_1^{-2}), C_{\max,s}h_2^d\left(1 + \Delta t^2 h_2^{-2}\right)\Bigr\}}{ \min \Bigl\{C_{\min,f}h_1^d(1+\Delta t), C_{\min,s} h_2^d(1+\Delta t^2) \Bigr\}}.
\end{align*}

We have shown that the largest and smallest singular values of $S$ can be bounded in four different cases, leading to four different bounds on the condition number of $S$.
\begin{enumerate}
       \item When  $C_{\max,f}h_1^d(1+\Delta t h_1^{-2}) \geq C_{\max,s}h_2^d\left(1 + \Delta t^2 h_2^{-2}\right)$ and $C_{\min,f}h_1^d(1+\Delta t) \leq  C_{\min,s} h_2^d(1+\Delta t^2)$,
       
       \begin{align*}k(S) &\leq \left(\frac{C_{\max} h_2^{(d-2)} (h_2^d + h_\gamma^{d-1})}{C_{\min}h_1^{2d}}\right)\frac{C_{\max,f}h_1^d(1+\Delta t h_1^{-2})}{ C_{\min,f}h_1^d(1+\Delta t)}  \\
       &= \overline{C} h_2^{d-2} h_1^{-2d}(h_2^d + h_{\gamma}^{d-1})\frac{1+\Delta t h_1^{-2}}{1+\Delta t}.
       \end{align*} 
       
       \item When  $C_{\max,f}h_1^d(1+\Delta t h_1^{-2}) \geq C_{\max,s}h_2^d\left(1 + \Delta t^2 h_2^{-2}\right)$ and $C_{\min,f}h_1^d(1+\Delta t) \geq  C_{\min,s} h_2^d(1+\Delta t^2)$,
       
       \begin{align*}k(S) &\leq \left(\frac{C_{\max} h_2^{(d-2)} (h_2^d + h_\gamma^{d-1})}{C_{\min}h_1^{2d}}\right)\frac{C_{\max,f}h_1^d(1+\Delta t h_1^{-2})}{ C_{\min,s}h_2^d(1+\Delta t^2)}  \\
       &= \overline{C} h_2^{-2} h_1^{-d}(h_2^d + h_{\gamma}^{d-1})\frac{1+\Delta t h_1^{-2}}{1+\Delta t^2}.
       \end{align*} 
       
        \item When  $C_{\max,f}h_1^d(1+\Delta t h_1^{-2}) \leq C_{\max,s}h_2^d\left(1 + \Delta t^2 h_2^{-2}\right)$ and $C_{\min,f}h_1^d(1+\Delta t) \leq  C_{\min,s} h_2^d(1+\Delta t^2)$,
       
       \begin{align*}k(S) &\leq \left(\frac{C_{\max} h_2^{(d-2)} (h_2^d + h_\gamma^{d-1})}{C_{\min}h_1^{2d}}\right)\frac{C_{\max,s}h_2^d(1+\Delta t^2 h_2^{-2})}{ C_{\min,f}h_1^d(1+\Delta t)}  \\
       &= \overline{C} h_2^{2d-2} h_1^{-3d}(h_2^d + h_{\gamma}^{d-1})\frac{1+\Delta t^2 h_2^{-2}}{1+\Delta t}.
       \end{align*} 

       \item When  $C_{\max,f}h_1^d(1+\Delta t h_1^{-2}) \leq C_{\max,s}h_2^d\left(1 + \Delta t^2 h_2^{-2}\right)$ and $C_{\min,f}h_1^d(1+\Delta t) \geq  C_{\min,s} h_2^d(1+\Delta t^2)$,
       
       \begin{align*}k(S) &\leq \left(\frac{C_{\max} h_2^{(d-2)} (h_2^d + h_\gamma^{d-1})}{C_{\min}h_1^{2d}}\right)\frac{C_{\max,s}h_2^d(1+\Delta t^2 h_2^{-2})}{ C_{\min,s}h_2^d(1+\Delta t^2)}  \\
       &= \overline{C} h_2^{d-2} h_1^{-2d}(h_2^d + h_{\gamma}^{d-1})\frac{1+\Delta t^2 h_2^{-2}}{1+\Delta t^2}.
       \end{align*} 

\end{enumerate}
In the case of equal meshes, where $h_1 = h_2 = h$, all four cases give:
\begin{align*}
      \kappa(S) &\leq \overline{C} (h^{-2} + h^{-d-2}h_{\gamma}^{d-1}) f(\Delta t,h,r,s), \hspace{5mm} \text{where } f(\Delta t,h,r,s) = \frac{1+\Delta t^r h^{-2}}{1+\Delta t^s}, \hspace{2mm} \text{ for } r,s \in \{1,2\}.
\end{align*}

\noindent We may always pick $\Delta t$ such that $\Delta t^r = h^2$, so that $f(\Delta t,h,r,s) \leq 2$. This gives
\begin{align*}
      \kappa(S) &\leq 2 \overline{C} (h^{-2} + h^{-d-2}h_{\gamma}^{d-1}). 
\end{align*}
Assuming $h_\gamma \approx K h$ gives
$\kappa(S) =  2\overline{C} (h^{-2} + h^{-3}).$
Overall, we have $\kappa(S) = O(h^{-3})$.

\begin{remark}
Note that system matrices for elliptic problems or linear elasticity scale as $h^{-2}$, \cite{Quarteroni_1994}.
For Schur complement matrices coupling two Stokes-like systems, one may expect the condition number to scale as $h^{-2}$ (\cite{Quarteroni_1994}, pp. 242, 303). Since our problem couples a parabolic and a hyperbolic PDE, it is reasonable that the Schur complement may include an extra factor of $h^{-1}$. This extra factor can be traced to our definition of the LM along the interface only, resulting in a factor of $h_\gamma^{d-1}$ instead of $h_\gamma^d$. 
\end{remark}

\section{Numerical Results.}\label{sec:Numerical}
We consider a manufactured solution to the FSI system \eqref{StokesMom}-\eqref{Inter2} for initial numerical results. On $\Omega_f = [0,1] \times [0,1]$ and $\Omega_s = [0,1] \times [1,2]$, define the velocity $\bm{u} = [u_1, u_2]^T$, pressure $p$, and displacement $\bm{\eta} = [\eta_1, \eta_2]^T$ as:
\begin{align*}
u_1  &= \text{cos (x+t) sin(y+t) + sin(x+t) cos(y+t)}  \\
u_2 &= \text{-cos(x+t) sin(y+t) - sin(x+t) cos(y+t)} \\
p &=  2\nu_f \text{(sin(x+t)sin(y+t) - cos(x+t)cos(y+t))}  + 2 \nu_s \text{cos(x+t) sin(y+t)} \\
\eta_1 &= \text{sin(x+t)sin(y+t)} \\
\eta_2 &= \text{cos(x+t) cos(x+t)}.
\end{align*}
Appropriate forcing functions $\bm{f_f}$, $\bm{f_s}$ and Neumann conditions $\bm{u_N}, \bm{\eta_N}$ are derived based on the FSI system defined in \eqref{StokesMom}-\eqref{Inter2}.
The $(P_2, P_1)$ pair is used for velocity and pressure, and $P_2$ is used for displacement. All constants ($\nu_f, \nu_s, \rho_f, \rho_s, \lambda$) are set to 1. We examine convergence in space and time, implementing Neumann boundary conditions on the right and left sides of $\Omega_f$, and Dirichlet elsewhere.


\begin{table}[h!]
\centering
\begin{tabular}{c|c|c|c|c|c|c|c|c|c|c}
 $\Delta x$     & $||\eta - \eta^h||_0$ & rate &  $||\eta - \eta^h||_1$ & rate &  $||u - u^h||_0$ & rate & $ ||u - u^h||_1$  & rate &  $||p - p^h||_0$ & rate \\
\hline
   1/2    &  1.936e-03 & -- & 2.967e-02 & -- & 2.538e-03 & -- & 3.822e-02 & -- & 2.266e-02 & -- \\
   1/4 & 2.421e-04 & 3.00 & 7.417e-03 & 2.00 & 3.203e-04 & 2.99 & 9.674e-03 & 1.98 & 3.848e-03 & 2.56 \\
   1/8 & 3.026e-05 & 3.00 & 1.854e-03 & 2.00 & 4.072e-05 & 2.98 & 2.462e-03 & 1.97 & 8.141e-04 & 2.24 \\
   1/16 & 3.783e-06 & 3.00 & 4.635e-04 & 2.00 & 5.162e-06 & 2.98 & 6.204e-04 & 1.99 & 1.969e-04 & 2.05 \\
   1/32 & 4.729e-07 & 3.00 & 1.159e-04 & 2.00 & 6.548e-07 & 2.98 & 1.555e-04 & 2.00 & 4.883e-05 & 2.01 \\
   1/64 & 5.956e-08 & 2.99 & 2.896e-05 & 2.00 & 8.544e-08 & 2.94 & 3.889e-05 & 2.00 & 1.219e-05 & 2.00
\end{tabular}
\caption{Convergence in space for $\Delta t = 10^{-5}$; $T=10^{-3}$}
 \label{table:SpaceConvMixedBCs}
\end{table}

\begin{table}[]
    \centering
    \begin{tabular}{c|c|c|c|c|c|c|c|c|c|c}
 $\Delta t$     & $||\eta - \eta^h||_0$ & rate &  $||\eta - \eta^h||_1$ & rate &  $||u - u^h||_0$ & rate & $ ||u - u^h||_1$  & rate &  $||p - p^h||_0$ & rate \\
 \hline
      1/4  & 7.599e-02 & -- & 3.715e-01 & -- & 1.372e-01 & -- & 5.281e-01 & -- & 3.018e-01 & -- \\
      1/8 & 4.104e-02 & 0.89 & 2.133e-01 & 0.80 & 7.888e-02 & 0.80 & 3.029e-01 & 0.80 & 1.755e-01 & 0.78 \\
      1/16 & 2.179e-02 & 0.91 & 1.185e-01 & 0.85 & 4.332e-02 & 0.86 & 1.664e-01 & 0.86 & 9.729e-02 & 0.85 \\
      1/32 & 1.141e-02 & 0.93 & 6.372e-02 & 0.89 & 2.300e-02 & 0.91 & 8.823e-02 & 0.92 & 5.179e-02 & 0.91 \\
      1/64 & 5.876e-03 & 0.96 & 3.331e-02 & 0.94 & 1.193e-02 & 0.95 & 4.571e-02 & 0.95 & 2.672e-02 & 0.95 \\
      1/128 & 2.990e-03 & 0.97 & 1.709e-02 & 0.96 & 6.079e-03 & 0.97 & 2.330e-02 & 0.97 & 1.355e-02 & 0.98
    \end{tabular}
    \caption{Convergence in time for $\Delta x = 1/32$; $T = 1$}
    \label{table:TimeConvMixedBCs}
\end{table}

As can be seen in Table \ref{table:SpaceConvMixedBCs} and Figure \ref{fig:FF_space_conv_mixed}, both the fluid velocity and structural displacement display approximately cubic convergence and quadratic convergence in the $L^2$ and $H^1$ norms, respectively. The pressure converges approximately quadratically, as well. We also see approximately linear convergence in time for all variables, as displayed in Table \ref{table:TimeConvMixedBCs} and  Figure \ref{fig:FF_time_conv_mixed}. 
Turning to Tables \ref{table:condSdeltxMixedBCs} and \ref{table:CondSDeltMixedBCs}, we observe that using the PCG method works as expected to decrease the condition number of the Schur complement matrix as well as the number of iterations required for solving the system. Table \ref{table:CondSDeltMixedBCs} shows $k(S)$ for a fixed $\Delta x$ and varied $\Delta t$; we include these results because although the Schur complement matrix is not time-dependent, it is dependent on the time-step $\Delta t$. We note that the highest condition numbers before preconditioning were only on the order of roughly $h^{-2}$, instead of the $h^{-3}$ bound proven. Perhaps in practice, the Schur complement matrix never reaches its maximum bounds. 

\begin{table}[]
\centering
\begin{tabular}{c|c|c|c|c}
 $\Delta x$   & CG: Iterations & PCG: Iterations & CG: k(S) & PCG: k(S)\\
 \hline
   1/2 & 15 & 6 & 17.45 & 10.54\\
   1/4 & 25 & 9 & 41.70 & 25.32\\
   1/8 & 39 & 13 & 131.78 & 63.87\\
   1/16 & 79 & 19 & 459.27 & 182.46\\
   1/32 & 150 & 26 & 1554.92 & 590.28\\
   1/64 & 274 & 34 & 4825.96 & 2091.33 
\end{tabular}
\caption{CG vs. PCG: iteration count and condition number of $S$; $\Delta t = 10^{-5}$; $T=10^{-3}$}
 \label{table:condSdeltxMixedBCs}
\end{table}
\begin{table}[h!]
\centering
\begin{tabular}{c|c|c}
 $\Delta t$ &  CG & PCG  \\
 \hline
   1/4 &  1535.25 & 28.91  \\
   1/8 &  1452.32 & 45.50 \\
   1/16 & 1375.43 & 74.56 \\
   1/32 & 1305.25 & 126.46 \\
   1/64 & 1231.64 & 186.36 \\
   1/128 & 1075.49 & 246.27 \\
\end{tabular}
\caption{Condition number of $S$ for CG and PCG; $\Delta x = 1/32$; $T=1$}
 \label{table:CondSDeltMixedBCs}
\end{table}

\begin{figure}[h!]
    \centering
    \begin{subfigure}[b]{0.45\textwidth} 
    \centering 
    \includegraphics[width=\textwidth]{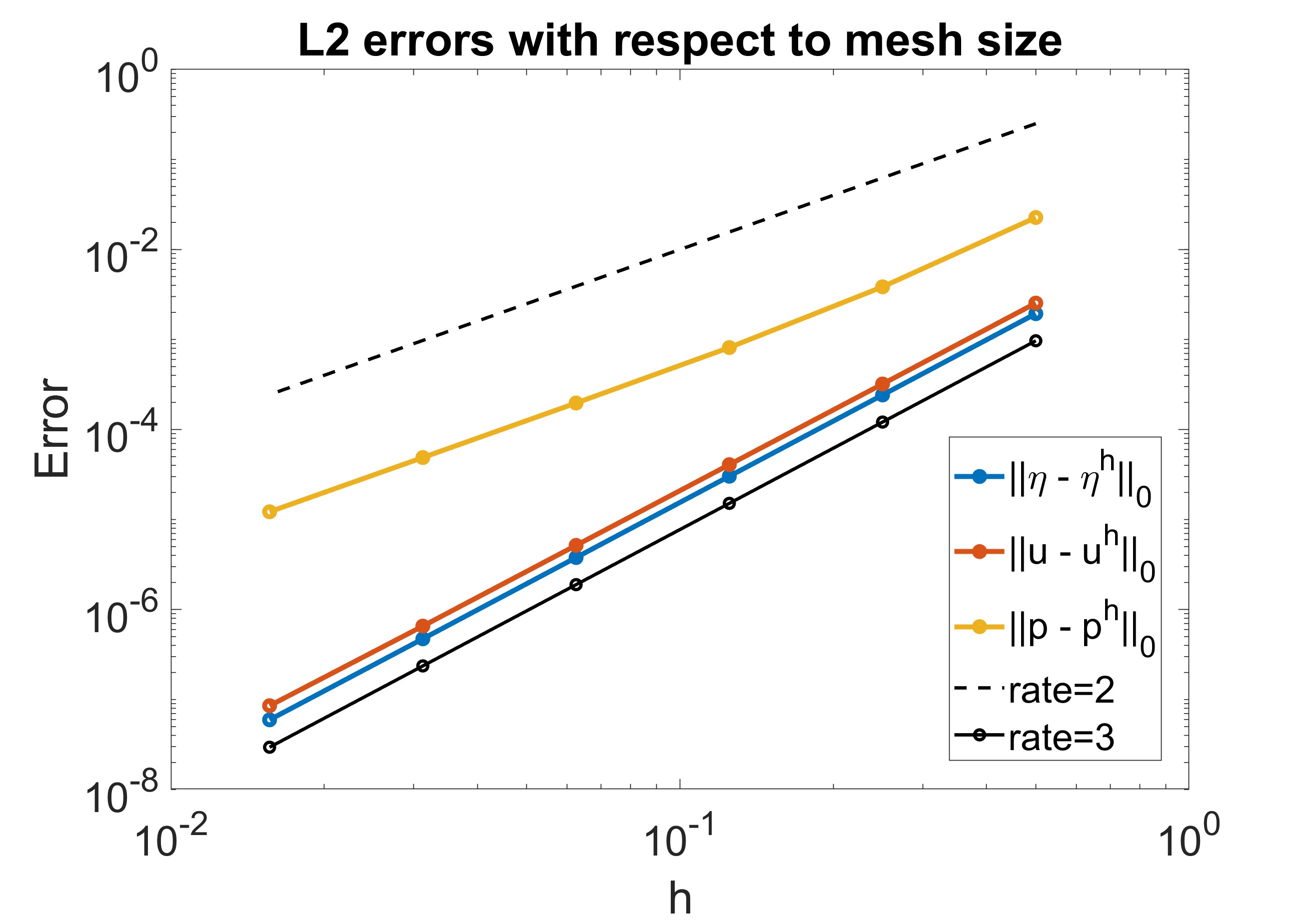}
        \caption{L2 errors}
        \label{fig:FF_L2_space_mixed}
    \end{subfigure}
     \begin{subfigure}[b]{0.45\textwidth}    
     \centering
     \includegraphics[width=\textwidth]{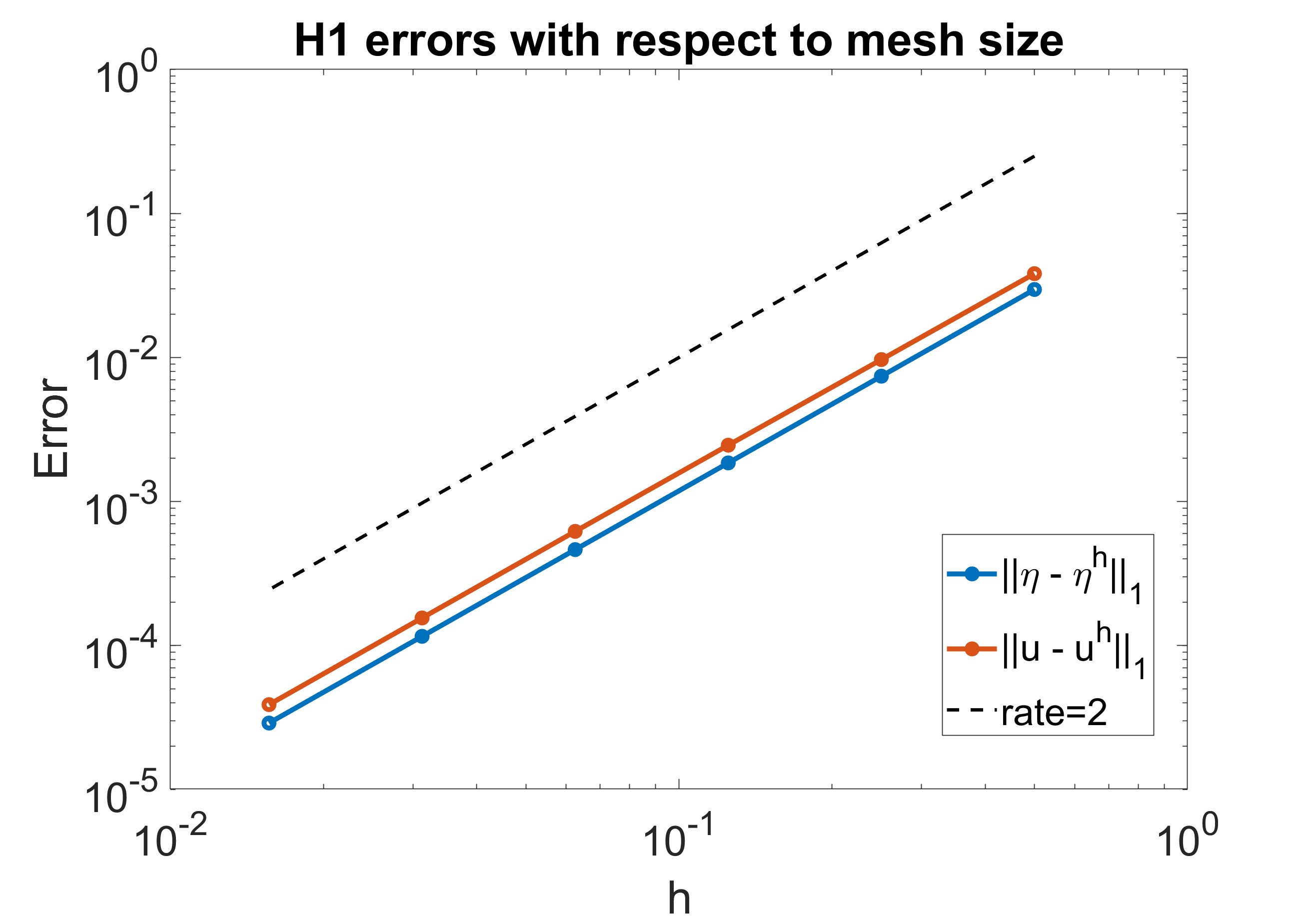}
        \caption{H1 errors}
        \label{fig:FF_H1_space_mixed}
    \end{subfigure}
    \caption{Convergence in space; $\Delta t=10^{-5}, T = 10^{-3}$}
    \label{fig:FF_space_conv_mixed}
\end{figure}

\begin{figure}[h!]
    \centering
    \begin{subfigure}[b]{0.45\textwidth} 
    \centering 
    \includegraphics[width=\textwidth]{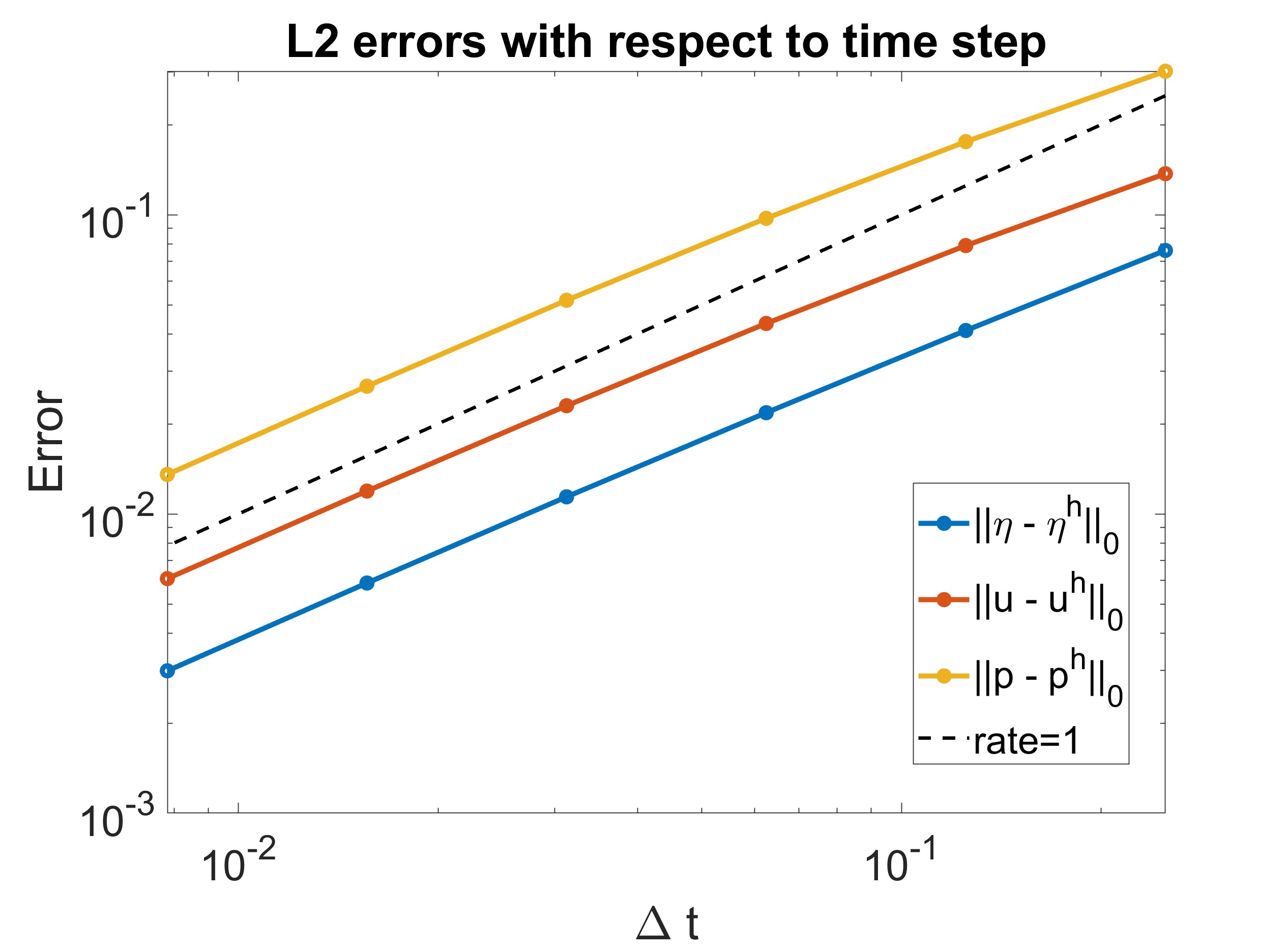}
        \caption{L2 errors}
        \label{fig:FF_L2_time_mixed}
    \end{subfigure}
     \begin{subfigure}[b]{0.45\textwidth}    
     \centering
     \includegraphics[width=\textwidth]{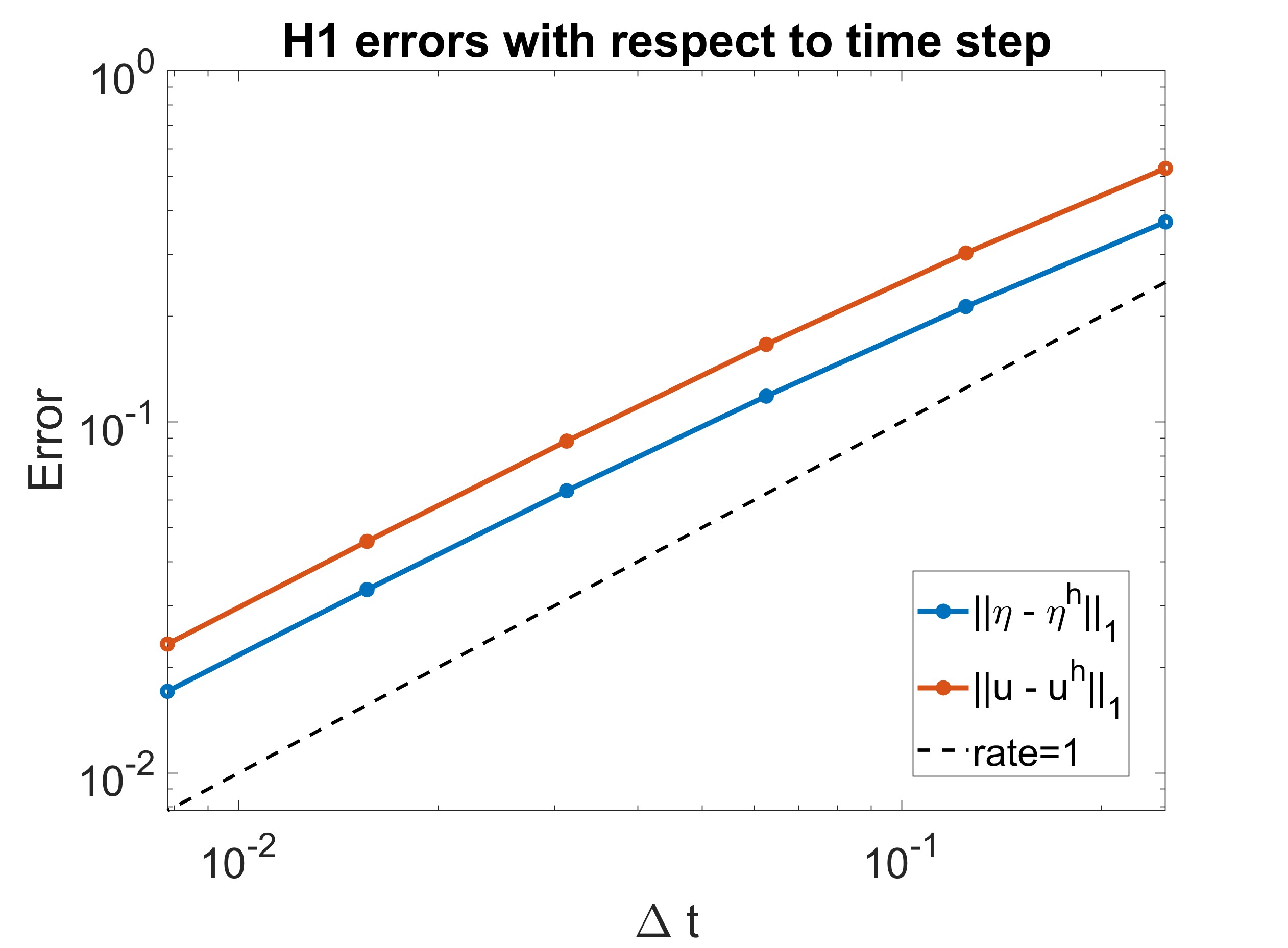}
        \caption{H1 errors}
        \label{fig:FF_H1_time_mixed}
    \end{subfigure}
    \caption{Convergence in time; $\Delta x=1/32, T = 1$}
    \label{fig:FF_time_conv_mixed}
\end{figure}

\section{Conclusions.}\label{sec:Conclusion}

We have presented a non-iterative, strongly coupled partitioned method for FSI problems. This method is centered on solving a Schur complement equation which implicitly expresses a Lagrange multiplier, representing interface flux and fluid pressure, in terms of fluid velocity and structural displacement. 
We have rigorously demonstrated the inf-sup conditions for both semi-discrete and fully discrete cases, and our analysis confirms that this formulation is well-posed. Additionally, initial numerical results demonstrate expected rates of convergence.

We plan to expand our approach by applying reduced order models (ROMs) to any variable involved in the coupling scheme. This technique will prove especially advantageous in scenarios that require multiple queries such as design or optimization, real-time parameter estimation, or control problems. 
By utilizing ROMs, we can significantly reduce the computational size and cost of the system and enhance the efficiency of our approach.

Furthermore, we plan to explore a more physically realistic situation by incorporating the Navier-Stokes equations instead of the linear Stokes equations for the fluid. We also intend to investigate a moving domain problem, which will provide additional insights into the applicability and performance of our method.

\vspace{.2in}

{\bf Appendix}

\vspace{.2in}

\noindent
{\bf Proof of Lemma \ref{lem:divEigBound}}
\begin{proof}
Let $\{\phi_i\}_i^{N}$ be a basis for the finite element space, and define the function $v^h = \sum_{i=1}^N v_i \phi_i$. The elements of the finite element divergence matrix satisfy $(L)_{ij} = \lambda (\nabla \cdot \phi_j, \nabla \cdot \phi_i)_{\Omega}$. Thus, we have $(L \bm{v}, \bm{v}) = \lambda(\nabla \cdot v^h,\nabla \cdot v^h)$. 
    We bound the term $\frac{(L\bm{v},\bm{v})}{|\bm{v}|^2}$.
$$   
         \frac{(L\bm{v},\bm{v})}{|\bm{v}|^2} = \frac{\lambda (\nabla \cdot v^h, \nabla \cdot v^h)}{|\bm{v}|^2}
            = \frac{\lambda ||\nabla \cdot v^h||_0^2}{|\bm{v}^h|^2} 
            \ge 0 \hspace{10mm} \text{since $\lambda > 0$}.
$$
Using the inverse inequality and Lemma \ref{lem:normFEandvec},
\begin{align*}
        \begin{split}
            \frac{(L\bm{v},\bm{v})}{|\bm{v}|^2} & = \,\frac{\lambda ||\nabla \cdot v^h||_0^2}{|\bm{v}|^2} \leq \frac{\lambda d ||\nabla v^h||_0^2}{|\bm{v}|^2}
                   \leq \,\lambda d \frac{C^* h^{-2}||v^h||_0^2}{|\bm{v}|^2} 
            \, \leq \, \lambda d C^* h^{-2} C_u h^d .
         \end{split}
     \end{align*}
   Combining, we see that 
     \begin{align*}
         0 \,\leq\, \frac{(L\bm{v},\bm{v})}{|\bm{v}|^2} \,\leq \,\lambda d C_u C^* h^{d-2}. 
     \end{align*}
     Since this holds for all vectors $\bm{v}$, in particular if $\bm{v}$ is an eigenvector of $L$, we have shown that an eigenvalue $\mu_i(L)$ of $L$ can be bounded as follows.
        \begin{align*}
         0 \,\leq \,\mu_i(L)\, \leq\, \lambda d C_u C^* h^{d-2}.
     \end{align*}
\end{proof}
\vspace{.2in}

\noindent
{\bf Proof of Lemma \ref{lem:condWBound}}
\begin{proof}
First, we use the following result from \cite{Ern_2004} to bound the eigenvalues of the mass matrix $M$. 

{\it For a quasi-uniform family of  triangularization $\mathcal{T}_h$ of $\Omega$, there exist constants $c_1, c_2$ independent of $h$ such that the eigenvalues of the mass matrix $M$ satisfy 
\begin{align*}
    c_1 h^d \leq \mu_i(M) \leq c_2 h^d. 
\end{align*}
}
For the FSI model, we have 
\begin{align}
    c_1 \rho_f h_1^d &\leq \mu_i(M_f) \leq c_2 \rho_f h_1^d \label{Mfbound},\\
    c_1 \rho_s h_2^d &\leq \mu_i(M_s) \leq c_2 \rho_s h_2^d. \label{Msbound}
\end{align}

Since both $W_{f}, W_{s}$ are symmetric positive definite matrices, their singular values $\sigma(W_r)$ are equivalent to their eigenvalues $\mu(W_r)$ for $r \in \{f,s\}$. With $W_{f} = M_f + \Delta t K_f$, the eigenvalues of $W_{f}$ can be bounded as follows, using \eqref{Mfbound} and Lemma \ref{lem:stiffEigBound}.
\begin{align*}
    c_1 \rho_f h_1^d + \Delta t (2 \nu_f \alpha C_\ell   h_1^d) &\leq \mu_i(W_{f}) \leq c_2 \rho_f h_1^d + \Delta t \left( 2 \nu_f \alpha^* C_u h_1^d (C^* h_1^{-2} + 1) \right), \\
    C_{\min,f} h_1^d (1 + \Delta t) &\leq \mu_i(W_{f}) \leq C_{\max,f} h_1^d (1 + \Delta t h_1^{-2}).
\end{align*}
Note that this result can also be found
in \cite{Quarteroni_1999} (see Eq. 7.1.18).
To bound the eigenvalues of $W_{s} =  M_s + \Delta t^2 (K_s + L)$, we combine Lemmas \ref{lem:stiffEigBound}, \ref{lem:divEigBound}, and \eqref{Msbound} to get the following result:
\begin{align*}
     c_1 \rho_s h_2^d + \Delta t^2 \left(2 \nu_s \alpha C_\ell  h_2^d + 0 \right) &\leq \mu_i(W_{s}) \leq  c_2 \rho_s h_2^d + \Delta t^2 \left( 2 \nu_s \alpha^* C_u h_2^d \left( C^* h_2^{-2} + 1 \right) + d \lambda C_u C^* h^{d-2}_2 \right), \\
    C_{\min,s} h_2^d \left(1 + \Delta t^2 \right) &\leq \mu_i(W_{s}) \leq  C_{\max,s} h_2^d \left(1 + \Delta t^2 h_2^{-2} \right).
\end{align*}

The eigenvalues of the block matrix $W$ are the union of the set of eigenvalues of $W_{f}$ and $W_{s}$.
Formally,
\begin{align*}
    \min \Bigl\{C_{\min,f}h_1^d(1+\Delta t), C_{\min,s} h_2^d(1+\Delta t^2) \Bigr\} \leq \mu_i(W) \leq \max\Bigl\{C_{\max,f}h_1^d(1+\Delta t h_1^{-2}), C_{\max,s}h_2^d\left(1 + \Delta t^2 h_2^{-2}\right)\Bigr\}.
\end{align*}

Thus, 
\begin{align}\label{eq:condW-1}
k(W) &= \frac{\mu_{\max}(W)}{\mu_{\min}(W)} 
\leq \frac{\max\Bigl\{C_{\max,f}h_1^d(1+\Delta t h_1^{-2}), C_{\max,s}h_2^d\left(1 + \Delta t^2 h_2^{-2}\right)\Bigr\}}{ \min \Bigl\{C_{\min,f}h_1^d(1+\Delta t), C_{\min,s} h_2^d(1+\Delta t^2) \Bigr\}}.
\end{align}
\end{proof}

\vspace{.2in}

\noindent
{\bf{Proof of Lemma \ref{lemma:AsigmaBound}}}
\begin{proof}
First, note that for $\bm{y} = \begin{bmatrix}
    \bm{u} \\ \bm{\eta} 
\end{bmatrix} \in Y $ and $\bm{z} = \begin{bmatrix}
    \bm{p} \\ \bm{g}
\end{bmatrix} \in Z$, we may write 
\begin{align}
    \bm{y}^T A^T \bm{z} = b(\bm{y}^h; \bm{z}^h).
\end{align}
Using \eqref{eq:variationalSigma}, 
the variational form for the singular values of $A$ or $A^T$ is as follows.
\begin{align}\label{eq:variationalSigmaforA}
    \sigma_{max}(A^T) = \underset{\bm{z} \in \mathbb{R}^{N_p + 2N_\gamma}}{\text{max}} \underset{\bm{y} \in \mathbb{R}^{2N_u + 2N_\eta}}{\text{max}} \frac{\bm{y}^T A^T \bm{z}}{|\bm{y}| |\bm{z}|}, \hspace{5mm}
     \sigma_{min}(A^T) = \underset{\bm{z} \in \mathbb{R}^{N_p + 2N_\gamma}}{\text{min}} \underset{\bm{y} \in \mathbb{R}^{2N_u + 2N_\eta}}{\text{max}} \frac{\bm{y}^T A^T \bm{z}}{|\bm{y}| |\bm{z}|}.
\end{align}

First, we need the continuity of the bilinear form $b(\cdot;\cdot)$.
\begin{align}\label{eq:bContinuity}
\begin{split}
    b(\bm{u}^h,\bm{\eta}^h;q^h,\bm{s}^h) &= (\bm{\eta}^h,\bm{s}^h)_\gamma - (\bm{u}^h,\bm{s}^h)_\gamma - (\nabla \cdot \bm{u}^h,q^h)_{\Omega_f} \\
    & \leq || \bm{\eta}^h ||_{1/2,\gamma} ||\bm{s}^h||_{-1/2,\gamma} +  ||\bm{u}^h||_{1/2,\gamma} ||\bm{s}^h||_{-1/2,\gamma} + ||\nabla \cdot \bm{u}^h||_{0,\Omega_f} ||q^h||_{0,\Omega_f} \\
    & \leq C \Big(||\bm{\eta}^h||_{1,\Omega_s} +||\bm{u}^h||_{1,\Omega_f}\Big) ||\bm{s}^h||_{-1/2,\gamma} +  ||\bm{u}^h||_{1,\Omega_f} ||q^h||_{0,\Omega_f} \\
    & \leq C \Big(||\bm{\eta}^h||_{1,\Omega_s} + ||\bm{u}^h||_{1,\Omega_f}\Big) ||\bm{s}^h||_{-1/2,\gamma} +||\bm{u}^h||_{1,\Omega_f} ||q^h||_{0,\Omega_f} + ||\bm{\eta}^h||_{1,\Omega_s} ||q^h||_{0,\Omega_f}\\ 
    &= C(||\bm{\eta}^h||_{1,\Omega_s} + ||\bm{u}^h||_{1,\Omega_f}) (||q^h||_{0,\Omega_f} + ||\bm{s}^h||_{-1/2,\gamma}) \\
    & \leq 2C \Big( ||\bm{\eta}^h||^2_{1,\Omega_s} + ||\bm{u}^h||^2_{1,\Omega_f} \Big)^{1/2} \Big( ||q^h||^2_{0,\Omega_f} + ||\bm{s}^h||^2_{-1/2,\gamma}\Big)^{1/2} \\
    &= 2C ||\bm{y}^h||_Y ||\bm{z}^h||_Z.
    \end{split}
\end{align}

By the inverse inequality $( ||\nabla \bm{v}^h||_0^2 \leq C^* h^{-2} ||\bm{v}^h||_0^2 \text{ for } \bm{v}^h \in V^h)$ and Lemma \ref{lem:normFEandvec} 
\begin{align}\label{eq:yhBound}
\begin{split}
    \frac{||\bm{y}^h||_Y^2}{|\bm{y}|^2} &= \frac{||\bm{u}^h||_{1,\Omega_f}^2 + ||\bm{\eta}^h||_{1,\Omega_s}^2}{|\bm{u}|^2 + |\bm{\eta}|^2} 
     \leq C_{PF}^{-2} \left( \frac{||\nabla \bm{u}^h||_{0,\Omega_f}^2}{|\bm{u}|^2} + \frac{|| \nabla \bm{\eta}^h||_{0,\Omega_s}^2}{|\bm{\eta}|^2} \right)  \\ 
    & \leq C_{PF}^{-2} \left( \frac{C^* h_1^{-2}||\bm{u}^h||_{0,\Omega_f}^2}{|\bm{u}|^2} + \frac{C^* h_2^{-2}|| \bm{\eta}^h||_{0,\Omega_s}^2}{|\bm{\eta}|^2} \right) \hspace{10mm}  \\
    & \leq C_{PF}^{-2} \left(  C^* h_1^{-2} C_u h_1^d + C^* h_2^{-2} C_u h_2^d \right)\hspace{10mm}  \\
    &= C_{PF}^{-2}C^* C_u (h_1^{d-2} + h_2^{d-2}).
    \end{split}
\end{align}
If we used the $Y$ norms on the interface $\gamma$ instead of the subdomains $\Omega_f, \Omega_s$, the exponents on $h_1, h_2$ would be $d-3$ instead of $d-2$, owing to the use of $h_i^{d-1}$ when applying Lemma \ref{lem:normFEandvec}. 
Next, 
\begin{align}\label{eq:zhBound}
\begin{split}
    \frac{||\bm{z}^h||_Z^2}{|\bm{z}|^2} &= \frac{||\bm{p}^h||_{0,\Omega_f}^2 + ||\bm{g}^h||_{-1/2,\gamma}^2}{|\bm{p}|^2 + |\bm{g}|^2} 
 \, \leq \, \frac{||\bm{p}^h||_{0,\Omega_f}^2}{|\bm{p}|^2} + \frac{||\bm{g}^h||_{-1/2,\gamma}^2}{|\bm{g}|^2} \\
    & \leq\frac{||\bm{p}^h||_{0,\Omega_f}^2}{|\bm{p}|^2} + \frac{||\bm{g}^h||_{0,\gamma}^2}{|\bm{g}|^2} 
\, \leq \, C_u (h_1^d + h_\gamma^{d-1}).  
    \end{split}
\end{align}
Here, we note that the power of $h_\gamma$ is one dimension less than the domain, since the LM is defined only along the interface and not the entire domain. This will contribute to the overall conditioning of the Schur complement matrix. 

Next, we may combine these to bound $\sigma_{max}(A^T)$. Without loss of generality, we assume here that $h_1 \leq h_2$. The continuity of $b(\cdot;\cdot)$ in \eqref{eq:bContinuity} and estimates  \eqref{eq:yhBound} and  \eqref{eq:zhBound} imply that 
\begin{align}
    \begin{split}
\frac{\bm{y}^T A^T \bm{z}}{|\bm{y}| |\bm{z}|} &= \frac{b(\bm{y}^h;\bm{z}^h)}{|\bm{y}| |\bm{z}|}  \leq 2C \frac{ ||\bm{y}^h||_Y ||\bm{z}^h||_Z}{|\bm{y}| |\bm{z}|}  \\ 
& \leq 2C \sqrt{C_{PF}^{-2} C^* C_u (h_1^{d-2} + h_2^{d-2})} \sqrt{C_u (h_1^d + h_\gamma^{d-1})}  \\
&\leq C_{\text{max}} h_2^{(d-2)/2}(h_2^d + h_\gamma^{d-1})^{1/2}.
    \end{split}
\end{align}
Using the characterization of $\sigma_{\max}(A^T)$ from \eqref{eq:variationalSigmaforA}, we see $\sigma_{\max}(A^T) \leq C_{\text{max}} h_2^{(d-2)/2}(h_2^d + h_\gamma^{d-1})^{1/2}$, where $C_{\max} = 2^{3/2}C C_{PF}^{-1}(C^*)^{1/2}C_u$.

Next, we wish to bound $\sigma_{\min}(A)$ from below. Again, WLOG, we assume $h_1 \leq h_2$.

\begin{align*}
    \begin{split}
\frac{||\bm{y}^h||^2_Y ||\bm{z}^h||^2_Z}{|\bm{y}|^2 |\bm{z}|^2} &= \left( \frac{||\bm{u}^h||_{1,\Omega_f}^2 + ||\bm{\eta}^h||_{1,\Omega_s}^2}{|\bm{u}|^2 + |\bm{\eta}|^2}\right) \left( \frac{||\bm{p}^h||_{0,\Omega_f}^2 + ||\bm{g}^h||_{-1/2,\gamma}^2}{|\bm{p}|^2 + |\bm{g}|^2} \right) \\ 
& \geq \left( \frac{||\bm{u}^h||_{0,\Omega_f}^2 + ||\bm{\eta}^h||_{0,\Omega_s}^2}{|\bm{u}|^2 + |\bm{\eta}|^2}\right) \left( \frac{||\bm{p}^h||_{0,\Omega_f}^2 + h_\gamma C_2^{-1}||\bm{g}^h||_{1/2,\gamma}^2}{|\bm{p}|^2 + |\bm{g}|^2} \right)  \hspace{5mm} \text{ by \eqref{InvIneq}}\\ 
& \geq \left( \frac{||\bm{u}^h||_{0,\Omega_f}^2 + ||\bm{\eta}^h||_{0,\Omega_s}^2}{|\bm{u}|^2 + |\bm{\eta}|^2}\right) \left( \frac{||\bm{p}^h||_{0,\Omega_f}^2 + h_\gamma C_2^{-1}||\bm{g}^h||_{0,\gamma}^2}{|\bm{p}|^2 + |\bm{g}|^2} \right) \\
& \geq \left( \frac{C_\ell h_1^d |\bm{u}|^2 + C_\ell h_2^d |\bm{\eta}|^2}{|\bm{u}|^2 + |\bm{\eta}|^2}\right) \left( \frac{C_\ell h_1^d |\bm{p}|^2 + h_\gamma C_2^{-1} C_\ell h_\gamma^{d-1} |\bm{g}|^2}{|\bm{p}|^2 + |\bm{g}|^2} \right) .
    \end{split}
\end{align*}

As in the inf-sup proof, assume that $\frac{h_\gamma}{h_1}, \frac{h_\gamma}{h_2}$ are large enough, i.e., there exists a constant $K$ independent of $h_i$ such that $h_\gamma \geq K h_2 \geq K h_1$. 
Then we may simplify the above inequality.
$$
        \frac{||\bm{y}^h||^2_Y ||\bm{z}^h||^2_Z}{|\bm{y}|^2 |\bm{z}|^2} \geq \left(C_\ell h_1^d\right) \left( \frac{C_\ell h_1^d \left( |\bm{p}|^2 + C_2^{-1} K |\bm{g}|^2 \right)}{|\bm{p}|^2 + |\bm{g}|^2} \right) \\
        = C_\ell^2 h_1^{2d} \overline{C}.
$$
Now, 
$$
    \frac{\bm{y}^T A^T \bm{z}}{|\bm{y}||\bm{z}|} = \frac{b(\bm{y}^h;\bm{z}^h)}{|\bm{y}||\bm{z}|} 
    \geq C_\ell h_1^d \overline{C}^{1/2}\frac{b(\bm{y}^h;\bm{z}^h)}{||\bm{y}^h||_Y ||\bm{z}^h||_Z}.
$$
Using the variational form for $\sigma_{\min}(A^T)$ in \eqref{eq:variationalSigmaforA} as well as the inf-sup condition in Theorem \eqref{thm:DiscInfSup} gives

\begin{align}
\begin{split}
\sigma_{min}(A^T) &= \underset{\bm{z} \in \mathbb{R}^{N_p + 2N_\gamma}}{\text{min}} \underset{\bm{y} \in \mathbb{R}^{2N_u + 2N_\eta}}{\text{max}} \frac{\bm{y}^T A^T \bm{z}}{|\bm{y}| |\bm{z}|} \geq \left( C_\ell \overline{C}^{1/2} h_1^d\right) \underset{\bm{z}^h \in Z}{\text{inf}} \hspace{2mm}\underset{\bm{y}^h \in Y}{\text{sup}} \frac{b(\bm{y}^h;\bm{z}^h)}{||\bm{y}^h||_Y ||\bm{z}^h||_Z}\\
&= C_{\min}  h_1^d.
\end{split}
\end{align}

\end{proof}

\small 
 \bibliography{main}

\begin{thebibliography}{10}

\bibitem{Ballarin_2017}
F.~Ballarin, G.~Rozza, and Y.~Maday.
\newblock Reduced-order semi-implicit schemes for fluid-structure interaction
  problems.
\newblock In P.~Benner, M.~Ohlberger, A.~Patera, G.~Rozza, and K.~Urban,
  editors, {\em Model {R}eduction of {P}arameterized {S}ystems}, pages
  149--167. Springer, Cham, 2017.

\bibitem{Boffi_2015}
D.~Boffi, N.~Cavallini, and L.~Gastaldi.
\newblock The finite element immersed boundary method with distributed
  {L}agrange multiplier.
\newblock {\em SIAM Journal on Numerical Analysis}, 53(6):2584--2604, 2015.

\bibitem{Boffi_2022_Parallel1}
D.~Boffi, F.~Credali, L.~Gastaldi, and S.~Scacchi.
\newblock A parallel solver for fluid structure interaction problems with
  {L}agrange multiplier.
\newblock {\em arXiv preprint arXiv:2212.13410}, 2022.

\bibitem{Boffi_2023_Parallel2}
D.~Boffi, F.~Credali, L.~Gastaldi, and S.~Scacchi.
\newblock A parallel solver for {FSI} problems with fictitious domain approach.
\newblock {\em Mathematical and Computational Applications}, 28:59, 2023.

\bibitem{Boffi_2017}
D.~Boffi and L.~Gastaldi.
\newblock A fictitious domain approach with {L}agrange multiplier for
  fluid--structure interactions.
\newblock {\em Numerische Mathematik}, 135:711--732, 2017.

\bibitem{Boffi_2022}
D.~Boffi and L.~Gastaldi.
\newblock Existence, uniqueness, and approximation of a fictitious domain
  formulation for fluid-structure interactions.
\newblock {\em Rendiconti Lincei}, 33:109--137, 2022.

\bibitem{Braess_1997}
D.~Braess and L.~Schumaker.
\newblock {\em Finite {E}lements: {T}heory, {F}ast {S}olvers, and
  {A}pplications in {S}olid {M}echanics}.
\newblock Cambridge University Press, 1997.

\bibitem{Burman_2009}
E.~Burman and M.A. Fern{\'a}ndez.
\newblock Stabilization of explicit coupling in fluid-structure interaction
  involving fluid incompressibility.
\newblock {\em Computer Methods in Applied Mechanics and Engineering}, 198
  (5-8):766--784, 2009.
\newblock inria-00247409v4.

\bibitem{Causin_2006}
P.~Causin, J.~Gerbeau, and F.~Nobile.
\newblock Added-mass effect in the design of partitioned algorithms for
  fluid-structure problems.
\newblock Technical Report RR-5084, INRIA, 2004.
\newblock inria-00071499.

\bibitem{Charina_2004}
M.~Charina, A.J. Meir, and P.G. Schmidt.
\newblock Mixed velocity, stress, current, and potential boundary conditions
  for stationary {MHD} flow.
\newblock {\em Computers and Mathematics with Applications}, 48:1181--1190,
  2004.

\bibitem{Chorin_1968}
A.~Chorin.
\newblock Numerical solution of the {N}avier-{S}tokes equations.
\newblock {\em Mathematics of Computation}, 22:745--762, 1968.

\bibitem{de_Castro_2022}
A.~de~Castro, P.~Kuberry, I.~Tezaur, and P.~Bochev.
\newblock A novel partitioned approach for reduced order model -- finite
  element model ({ROM-FEM}) and {ROM-ROM} coupling.
\newblock In {\em Earth and Space 2022}, pages 475--489. 2022.

\bibitem{Du_2003}
Q.~Du, M.D. Gunzburger, L.S. Hou, and J.~Lee.
\newblock Analysis of a linear fluid-structure interaction problem.
\newblock {\em Discrete and Continuous Dynamical Systems}, 9:633--650, 2003.

\bibitem{Ern_2004}
A.~Ern and J.-L. Guermond.
\newblock {\em Theory and {P}ractice of {F}inite {E}lements, in: {A}pplied
  {M}athematical {S}ciences, no. 159}.
\newblock Springer-Verlag, 2004.

\bibitem{Farhat_1994}
C.~Farhat, L.~Crivelli, and F-X. Roux.
\newblock A transient {FETI} methodology for large-scale parallel implicit
  computations in structural mechanics.
\newblock {\em International Journal for Numerical Methods in Engineering},
  37(11):1945--1975, 1994.

\bibitem{Farhat_1991}
C.~Farhat and F-X. Roux.
\newblock A method of finite element tearing and interconnecting and its
  parallel solution algorithm.
\newblock {\em International Journal for Numerical Methods in Engineering},
  32(6):1205--1227, 1991.

\bibitem{Fernandez_2011}
M.~Fern{\'a}ndez.
\newblock Coupling schemes for incompressible fluid-structure interaction:
  implicit, semi-implicit and explicit.
\newblock {\em SeMA Journal: Boletin de la Sociedad Espa\~{n}ola de
  Matem{'a}tica Aplicada}, pages 59--108, 2011.
\newblock inria-00580772v2.

\bibitem{Fernandez_2005}
M.~Fern{\'a}ndez, J.~Gerbeau, and C.~Grandmont.
\newblock A projection semi-implicit scheme for the coupling of an elastic
  structure with an incompressible fluid.
\newblock Technical Report RR-5700, INRIA, 2005.
\newblock inria-00070315.

\bibitem{Gerbeau_2003}
J.~Gerbeau and M.~Vidrascu.
\newblock A {Q}uasi-{N}ewton algorithm based on a reduced model for
  fluid-structure interaction problems in blood flows.
\newblock {\em ESAIM: M2AN}, 37:631--647, 2003.

\bibitem{Gerstenberger_2008}
A.~Gerstenberger and W.~Wall.
\newblock An extended finite element method/{L}agrange multiplier based
  approach for fluid--structure interaction.
\newblock {\em Computer Methods in Applied Mechanics and Engineering},
  197:1699--1714, 2008.

\bibitem{Glowinski_2007}
R.~Glowinski and Y.~Kuznetsov.
\newblock Distributed {L}agrange multipliers based on fictitious domain method
  for second order elliptic problems.
\newblock {\em Computer Methods in Applied Mechanics and Engineering},
  196(8):1498--1506, 2007.

\bibitem{Glowinski_Dir_1994}
R.~Glowinski, T-W. Pan, and J.~Periaux.
\newblock A fictitious domain method for {D}irichlet problem and applications.
\newblock {\em Computer Methods in Applied Mechanics and Engineering},
  111:283--303, 1994.

\bibitem{Glowinski_NS_1994}
R.~Glowinski, T-W. Pan, and J.~Periaux.
\newblock A fictitious domain method for external incompressible viscous flow
  modeled by {N}avier--{S}tokes equations.
\newblock {\em Computer Methods in Applied Mechanics and Engineering},
  112:133--148, 1994.

\bibitem{Grigoriadis_2009}
D.~Grigoriadis, S.~Kassinos, and E.~Votyakov.
\newblock Immersed boundary method for the {MHD} flows of liquid metals.
\newblock {\em Journal of Computational Physics}, 228:903--920, 2009.

\bibitem{Gunzburger_1992}
M.~Gunzburger and S.~Hou.
\newblock Treating inhomogeneous essential boundary conditions in finite
  element methods and the calculation of boundary stresses.
\newblock {\em SIAM Journal of Numerical Analysis}, 29:390--424, 1992.

\bibitem{Hay_2015}
A.~Hay, S.~Etienne, A.~Garon, and D.~Pelletier.
\newblock Time-integration for {ALE} simulations of fluid–structure
  interaction problems: stepsize and order selection based on the {BDF}.
\newblock {\em Computer Methods in Applied Mechanics and Engineering},
  295:172--195, 2015.

\bibitem{He_2019}
Y.~He and J.~Shen.
\newblock Unconditionally stable pressure-correction schemes for a nonlinear
  fluid-structure interaction model.
\newblock {\em Communications on Applied Mathematics and Computation},
  1:61--–80, 2019.

\bibitem{Hou_2012}
G.N. Hou, J.~Wang, and A.~Layton.
\newblock Numerical {M}ethods for {F}luid-{S}tructure {I}nteraction - {A}
  {R}eview.
\newblock {\em Communications in Computational Physics}, 12:337--377, 2012.

\bibitem{Hubner_2004}
B.~H{\"u}bner, E.~Walhorn, and D.~Dinkler.
\newblock A monolithic approach to fluid–structure interaction using
  space–time finite elements.
\newblock {\em Computer Methods in Applied Mechanics and Engineering},
  193:2087–2104, 2004.

\bibitem{Kwak_2014}
J.~Kwak, T.~Chun, S.~Shin, and O.~Bauchau.
\newblock Domain decomposition approach to flexible multibody dynamics
  simulation.
\newblock {\em Computational Mechanics}, 53:147--158, 2014.

\bibitem{Lee_2020}
S-H Lee, Y.~Kim, and et~al. D.~Gong.
\newblock Fast and novel computational methods for multi-scale and
  multi-physics: {FETI} and {POD}-{ROM}.
\newblock {\em Multiscale Sci. Eng.}, 2:189--197, 2020.

\bibitem{Li_2012}
J.~Li, C.~Farhat, P.~Avery, and R.~Tezaur.
\newblock A dual-primal {FETI} method for solving a class of fluid--structure
  interaction problems in the frequency domain.
\newblock {\em International Journal for Numerical Methods in Engineering},
  89:418--437, 2012.

\bibitem{Marcsa_2013}
D.~Marcsa and M.~Kuczmann.
\newblock Finite element tearing and interconnecting method and its algorithms
  for parallel solution of magnetic field problems.
\newblock {\em Electrical, Control and Communication Engineering}, 3(1):25--30,
  2013.

\bibitem{Matthies_2002}
H.~Matthies and J.~Steindorf.
\newblock Partitioned strong coupling algorithms for fluid–structure
  interaction.
\newblock {\em Computers \& Structures}, 81(8):805--812, 2002.

\bibitem{Nonino_2021}
M.~Nonino, F.~Ballarin, and G.~Rozza.
\newblock A {M}onolithic and a {P}artitioned, {R}educed {B}asis {M}ethod for
  {F}luid–{S}tructure {I}nteraction {P}roblems.
\newblock {\em Fluids}, 6:229--263, 2021.

\bibitem{Park_2000}
K.~Park, C.~Felippa, and U.~Gumaste.
\newblock A localized version of the method of {L}agrange multipliers and its
  applications.
\newblock {\em Computational Mechanics}, 24:476--490, 2000.

\bibitem{Peterson_2019}
K.~Peterson, P.~Bochev, and P.~Kuberry.
\newblock Explicit synchronous partitioned algorithms for interface problems
  based on {L}agrange multipliers.
\newblock {\em Computers \& Mathematics with Applications}, 78:459--482, 2019.

\bibitem{Quaini_2009}
A.~Quaini.
\newblock {\em Algorithms for {F}luid-{S}tructure {I}nteraction {P}roblems
  {A}rising in {H}emodynamics}.
\newblock PhD thesis, {\'E}cole {P}olytechnique {F}{\'e}d{\'e}rale de
  {L}ausanne, 2009.

\bibitem{Quaini_2007}
A.~Quaini and A.~Quarteroni.
\newblock A semi-implicit approach for fluid-structure interaction based on an
  algebraic fractional step method.
\newblock {\em Mathematical Models and Methods in Applied Sciences},
  17:957--983, 2007.

\bibitem{Quarteroni_1994}
A.~Quarteroni and A.~Valli.
\newblock {\em Numerical {A}pproximation of {P}artial {D}ifferential
  {E}quations}.
\newblock Springer-Verlag, 1994.

\bibitem{Quarteroni_1999}
A.~Quarteroni and A.~Valli.
\newblock {\em Domain {D}ecomposition {M}ethods for {P}artial {D}ifferential
  {E}quations}.
\newblock Oxford University Press, 1999.

\bibitem{Ross_2008}
M.R. Ross, C.A. Felippa, K.C. Park, and M.A. Sprague.
\newblock Treatment of acoustic fluid–structure interaction by localized
  {L}agrange multipliers: {F}ormulation.
\newblock {\em Computer Methods in Applied Mechanics and Engineering},
  197(33):3057--3079, 2008.

\bibitem{Sockwell_2020}
K.C. Sockwell, K.~Peterson, P.~Kuberry, P.~Bochev, and N.~Trask.
\newblock Interface {F}lux {R}ecovery coupling method for the
  ocean–atmosphere system.
\newblock {\em Results in Applied Mathematics}, 8:100--110, 2020.

\bibitem{Temam_1968}
R.~Temam.
\newblock Une m{\'e}thode d'approximation de la solution des {\'e}quations de
  {N}avier-{S}tokes.
\newblock {\em Bulletin de la Soci{\'e}t{\'e} Math{\'e}matique de France},
  96:115--152, 1968.

\bibitem{Temam_1984}
R.~Temam.
\newblock {\em Navier-{S}tokes Equations Theory and Numerical Analysis}.
\newblock {AMS} Chelsea Publishing, 1984.

\bibitem{Walsh_2004}
T.~Walsh, G.~Reese, K.~Pierson, H.~Sumali, J.~Dohner, and D.~Day.
\newblock Computational and experimental techniques for coupled
  acoustic/structure interactions.
\newblock Technical report, Sandia National Laboratories (SNL), Albuquerque,
  NM, and Livermore, CA, 2004.

\bibitem{Wang_2009}
J.~Wang and A.~Layton.
\newblock Numerical simulations of fiber sedimentation in {N}avier-{S}tokes
  flows.
\newblock {\em Communications in Computational Physics}, 5:61--83, 2009.

\bibitem{Zhang_2007}
W.~Zhang, Y.~Jiang, and Z.~Ye.
\newblock Two better loosely coupled solution algorithms of {CFD} based
  aeroelastic simulation.
\newblock {\em Engineering Applications of Computational Fluid Mechanics},
  1:253--262, 2007.

\end{thebibliography}
    \bibliographystyle{plainurl}

\end{document}